\documentclass[12pt]{amsart}
\usepackage[colorlinks,linkcolor=blue,citecolor=blue,urlcolor=blue, pdfcenterwindow, pdfstartview={XYZ null null 1.2}, pdffitwindow, pdfdisplaydoctitle=true]{hyperref}
\usepackage[english]{babel}
\usepackage{array}
\usepackage{threeparttable}
\usepackage{rotating}
\usepackage[section]{placeins}
\usepackage{amssymb}
\usepackage{amsmath}
\usepackage{mathrsfs}
\usepackage{mathtools}
\usepackage{enumerate}
\usepackage{enumitem}
\usepackage{color}
\usepackage{relsize}
\usepackage{amsrefs}        
\usepackage[all]{xy}
\usepackage{tikz-cd}
\usepackage{comment}

\input xy
\xyoption{all}
\begin{document}
\newtheorem{lemma}{Lemma}[section]
\newtheorem{lemm}[lemma]{Lemma}
\newtheorem{prop}[lemma]{Proposition}
\newtheorem{coro}[lemma]{Corollary}
\newtheorem{theo}[lemma]{Theorem}
\newtheorem{conj}[lemma]{Conjecture}
\newtheorem{prob}{Problem}
\newtheorem{ques}{Question}
\newtheorem{rema}[lemma]{Remark}
\newtheorem{rems}[lemma]{Remarks}
\newtheorem{defi}[lemma]{Definition}
\newtheorem{defis}[lemma]{Definitions}
\newtheorem{exam}[lemma]{Example}

\newcommand{\N}{\mathbf N}
\newcommand{\Z}{\mathbf Z}
\newcommand{\R}{\mathbf R}
\newcommand{\Q}{\mathbf Q}
\newcommand{\C}{\mathbf C}

\newcommand{\todo}[1]{\vspace{5mm}\par \noindent
\framebox{\begin{minipage}[c]{0.95 \textwidth} \tt #1
\end{minipage}} \vspace{5mm} \par}

\title[Discrete Laplace operators without eigenvalues]{Laplace and Schr\"odinger  operators without eigenvalues on homogeneous amenable graphs}
\date{December 18th, 2022}
\author{R. Grigorchuk}
\address{Texas A\&M University}
\email{grigorch@math.tamu.edu}

\author{Ch. Pittet}
\address{Aix-Marseille University CNRS I2M and University of Geneva}
\email{pittet@math.cnrs.fr}


\thanks{The authors acknowledge support of the FNS grant 200020-200400. The first author acknowledges partial support from the University of Geneva and from the Simons Foundation through Collaboration Grant 527814.}
\dedicatory{Dedicated to the memory of Mikhail A. Shubin (1944--2020)}

\begin{abstract}
A one-by-one exhaustion is a combinatorial/geometric sufficient condition for excluding finitely supported eigenfunctions
of Laplace and Schr\"odinger operators on graphs. Isoperimetric inequalities in graphs with a cocompact automorphism group
provide an upper bound on the von Neumann dimension of the space of $l^2$-eigenfunctions. Any finitely generated indicable amenable group has a Cayley graph without $l^2$-eigenvalues.
There exists a finitely generated group $G$ with finite generating sets $S$ and $S'$ such that the adjacency operator of the Cayley graph of $(G,S)$ has no $l^2$-eigenvalue while  the adjacency operator of the Cayley graph of $(G,S')$ has pure point $l^2$-spectrum.
\end{abstract}
\maketitle
\tableofcontents
\section{Introduction}
\subsection{Infinite connected graphs without eigenvalues} Let $\Gamma$ be a weighted connected graph  with  infinite vertex set (see  Subsection \ref{subsection: the path-metric on a connected graph}, and Subsection \ref{subsection: the Hilbert spaces associated to a weighted graph}). Let $\Delta$ be the associated  Laplacian on $\Gamma$ (see  Subsection 
\ref{Laplace}). Which geometric or combinatorial properties of $\Gamma$ imply that the $l^2$-spectrum of $\Delta$ contains no eigenvalue? (In this paper we consider exclusively the $l^2$-spectrum of Laplace and Schrödinger operators; when we speak about an eigenvalue, it is implicitly understood that it corresponds to  an $l^2$-eigenfunction.) A necessary and sufficient condition, based on Bloch analysis, is known for  graphs admitting a cocompact free action by a finitely generated abelian group of automorphisms, see \cite[Proposition 4.2]{HiguchiNomura}. See also \cite[Theorem 4]{Kuch}. Applying this condition, Higuchi and Nomura deduce that the combinatorial Laplacian (for the definition see the end of Subsection \ref{subsection: adjacency and  Markov operators}) on a graph which is the  maximal abelian covering of (the realization of) a connected finite graph having a $2$-factor, has no eigenvalue \cite[Theorem 2]{HiguchiNomura}.
(A finite connected graph $\Gamma$ has a $2$-factor if and only if there exists a finite number of oriented simplicial circles which disjointly embed in $\Gamma$ in such a way that any vertex of $\Gamma$ lies on one of the embedded circles. The maximal abelian cover of the geometric realization $T$ of a finite connected graph $\Gamma$ is the Galois cover of $T$ defined by the  commutator subgroup of the fundamental group of $T$, hence the Galois group is isomorphic to the  first homology group $H_1(T,\mathbb Z)$.)  With the help of these results, Higuchi and Nomura are able to decide, for several examples of planar graphs, wether the  Laplacian admits an eigenvalue or not, see \cite[6. Examples]{HiguchiNomura}. They also ask  for new geometric or combinatorial properties implying the NEP (no eigenvalue property) for graphs whose automorphism group contains a finite index finitely generated abelian subgroup \cite[Problem 6.11]{HiguchiNomura}. A similar question is raised in \cite[page 93]{HarRobVal} for a general finitely generated group. The spectrum of the  combinatorial Laplacian $\Delta$  has been studied  on  Cayley graphs (for the definition of a Cayley graph see Subsection \ref{subsection: the path-metric on a connected graph}) of some finitely generated metabelian groups by several authors (see \cite{BarWoe}, \cite{DicSch}, \cite{GriZuk}, \cite{LehNeuWoe}). We recall the example of the so called ``lamplighter group''. Consider the ring $\mathbb F_2[X,X^{-1}]$ of Laurent polynomials in the variable $X$ with coefficients in the field $\mathbb F_2$ with $2$ elements. Its group of units 
$$\mathbb F_2[X,X^{-1}]^{\times}=\{X^n: n\in\mathbb Z\}\cong\mathbb Z$$ 
acts by multiplication on $\mathbb F_2[X,X^{-1}]$. The lamplighter group $L$ is the corresponding semi-direct product
\[
	L=\mathbb F_2[X,X^{-1}]\rtimes \mathbb Z.
\]
The elements $a=(0,1)$ and $c=(1,0)$ together generate $L$ (right multiplication by $a$ increases the position of the lamplighter by $1$, right multiplication by $c$ switches the lamp the lamplighter stands at). Putting $b=ac$, it is obvious that the set $S'=\{a;b\}$ also generates $L$. In \cite{GriZuk} it is proved that the eigenvalues of the combinatorial Laplacian $\Delta$ of the Cayley graph $\mathcal C(L,S')$ form a dense countable subset of the spectrum $[0,2]$ of $\Delta$ and in \cite{DicSch} and \cite{BarWoe}, an orthonormal Hilbert basis of finitely supported eigenfunctions is constructed. The spectral properties of the lamplighter group have also been investigated in relation with the spectral properties of de Bruijn graphs and spider-web graphs, see \cite{GriLeeNag} and \cite{BalDha}.
Does the structure of the spectrum depend on the set of generators?  The question was brought up in \cite[page 210]{GriZuk}. A positive answer has been given by Grabowski and Virag in an unpublished preprint from 2015 entitled ``Random walks on Lamplighters via random Schr\"odinger operators''. Grabowski and Virag use the work of Martinelli and Micheli \cite{MarMic} to deduce that $L$ has a system of generators with singular continuous spectral measure; we refer the reader to \cite[page 655]{Gra} and to \cite[pages 2, 4, 22, 29]{GriSim} for more details.
The following theorem provides many examples of finitely generated amenable groups with Cayley graphs having no eigenvalues. (We refer the reader to \cite{Ger} and \cite{delaHarGriSil} for the definitions of amenable group and amenable graph.)

\begin{theo}\label{theorem: indicable}(Cayley graphs of indicable amenable groups without eigenvalue.) Let $G$ be a group with an epimorphism
	$$h:G\to\mathbb Z$$
from $G$ to $\mathbb Z$ the infinite cyclic group (in other words $G$ is indicable). 
Assume $G$ is finitely generated. Consider a finite generating set $S$ of $G$ of the form
\[
	S=\{t\}\bigcup K
\]
such that  $h(t)$ generates $\mathbb Z$ and $K\subset\emph{Ker}(h)$. Then the combinatorial Laplace operator on the Cayley graph of $G$ with respect to $S$ has no eigenfunction with finite support.	If moreover $G$ is amenable, then the combinatorial Laplace operator on the Cayley graph of $G$ with respect to $S$ has no eigenfunction.	
\end{theo}

\begin{rema} We would like to emphasize that it is  not required that $K$ generates $\emph{Ker}(h)$ (which is often not finitely generated). It is easy to check that in a finitely generating group $G$, with an epimorphism $h:G\to\mathbb Z$, there is always a generating set of the required form $S=\{t\}\bigcup K$.
\end{rema}

There are two main steps in the proof of Theorem \ref {theorem: indicable}. These two steps are carefully explained in the remaining part of this introduction, after the formal proof that we give now, assuming all the needed definitions and preliminary results.  

\begin{proof} First step: the homomorphism $h$ is a ``height function'' (see the hypothesis of Theorem \ref{theorem: exhaustion by level sets} for the properties of a height function) on the Cayley graph of $G$ with respect to $S$ hence Theorem \ref{theorem: exhaustion by level sets} applies (the special form of $S$ is essential here) and implies the $\lambda$-unicity for any finite set of vertices and any $\lambda\in\mathbb R$ (see Definition \ref{definition: uniqueness}). Second step: the amenability of $G$ is equivalent to the existence of a F{\o}lner sequence (see Definition \ref{definition: amenable graph}) in the Cayley graph, hence implication $(1)\implies (4)$ from Theorem \ref{theorem: general principles} is true and we conclude that
the $\lambda$-unicity for any finite set of vertices from step one implies that there is no square-summable eigenfunction.
\end{proof}
Combining the above Theorem with results of Grigorchuk and Zuk and Bartholdi and Woess we obtain the following corollary.

\begin{coro}\label{cor: lamplighter} Let $L=\mathbb F_2[X,X^{-1}]\rtimes \mathbb Z$ be the lamplighter group. We denote $a=(0,1)$, $c=(1,0)$, and $b=ac$. Consider the generating sets $S=\{a;c\}$ and $S'=\{a;b\}$.
	\begin{enumerate}
		\item The combinatorial Laplace operator on the Cayley graph of $(L,S)$  has no square-summable eigenfunction.
       \item  The space $l^2(L)$ admits an orthonormal Hilbert basis made of finitely supported eigenfunctions of the combinatorial Laplacian of the Cayley graph of $(L,S')$.
	   \end{enumerate}
	\end{coro}
\begin{proof} The projection
	\[
		h:\mathbb F_2[X,X^{-1}]\rtimes \mathbb Z\to\mathbb Z
	\]
sends $(0,1)$ to the generator $h(0,1)=1$ of the quotient $$\mathbb Z\cong\{X^n: n\in\mathbb Z\}$$ and $h(1,0)=0$ is the trivial element of $\mathbb Z\cong\{X^n: n\in\mathbb Z\}$. The group $L$ is metabelian hence amenable. Hence Theorem \ref{theorem: indicable} above applies and implies the statement about $(L,S)$. The statement about $(L,S')$ follows from the works of Grigorchuk and Zuk \cite{GriZuk} and Bartholdi and Woess \cite{BarWoe}.
\end{proof}
Theorem \ref{theorem: indicable} applies to several classes of groups (strongly polycyclic groups, free solvable groups, some HNN-extensions,  some wreath products, etc.) and is valid for a general class of operators described in Definition \ref{definition: Schrodinger operator}. In particular it applies to the operators on the lamplighter group considered by Virag and Grabowski \cite[page 655]{Gra}. Although the theorem says nothing about the singular spectrum, its conclusion is equivalent to the continuity of the integrated density of states (see the end of Subsection \ref{subsection: spectral projections}). 
Here are some examples. Take any finitely generated amenable indicable group $\mathfrak{G}$, choose a finite group $\mathfrak{H}$, form the wreath product $G=\mathfrak{H}\wr\mathfrak{G}$. According to Theorem \ref{theorem: indicable} there is a finitely generating set of $G$ such that the combinatorial Laplacian on the associated Cayley graph has no eigenvalue. According to \cite{LehNeuWoe} there exists a Cayley graph of $G$ with weights (defined by a symmetric probability measure on $G$) so that the spectrum of its associated Laplace operator admits a dense subset of eigenvalues.

The first step in the proof of  Theorem \ref{theorem: indicable} is to exclude the existence of a  finitely supported eigenfunction. Given a real number $\lambda$,
we say that the  Laplacian $\Delta$ of a weighted graph $\Gamma$ satisfies $\lambda$-uniqueness on a subset $\Omega$ of vertices of $\Gamma$ if  there is no $\lambda$-eigenfunction of $\Delta$ whose support is included in $\Omega$  (see Definition \ref{definition: uniqueness}). If a finite or infinite countable subset of vertices $\Omega$ admits a one-by-one exhaustion (see Definition \ref{definition: exhaustion}) then $\Delta$ satisfies $\lambda$-uniqueness on $\Omega$ for any $\lambda\in\mathbb R$ (see Theorem \ref{theorem: exhaustion implies lambda-uniqueness}). We emphasize that no hypothesis on the automorphism group of the graph is needed in this implication (no group is involved at this point). A one-by-one exhaustion of $\Omega$ can be understood as an inductive process. First we look for a vertex $v_1$ of the graph $\Gamma$ which is not in $\Omega$ and which has exactly one neighbor $w_1$ belonging to $\Omega$ and we remove $w_1$ from $\Omega$. Then we look  for a vertex $v_2$ not in $\Omega\setminus\{w_1\}$ which has exactly one neighbor $w_2$ belonging to $\Omega\setminus\{w_1\}$ and we remove $w_2$. And so on. If 
\[
	\Omega=\bigcup_n\{w_n\},
\] 
where the union is finite or infinite countable,
then we say that $\Omega$ admits a one-by-one exhaustion. For example, the combinatorial Laplacian of  the Cayley graph of $\mathbb Z\times\mathbb Z/2\mathbb Z$
(the direct product of the infinite cyclic group with the group of cardinality $2$) with respect to the generating set 
$
\left\{\left(1,\overline{0}\right);\left(0,\overline{1}\right)\right\}	
$
satisfies $\lambda$-uniqueness  on any finite subset because for any integer $n\geq 0$ the  subset $\Omega_n\subset G$ defined as
\[
	 \Omega_n=\{(k,\overline{x}): |k|\leq n, \overline{x}\in \mathbb Z/2\mathbb Z\}
\]
admits a one-by-one exhaustion. See Figure \ref{one} and  \ref{two}, which illustrate the case $n=1$. The dashed rectangle from Figure 
\ref{one} encloses the vertices of $\Omega_1$. The dashed polygon  from Figure \ref{two} encloses the vertices of $\Omega_1\setminus\{w_1\}$. The same subset $\Omega_1$ viewed in the Cayley graph of the same group but with respect to the generating set
\[
	\left\{\left(1,\overline{0}\right);\left(0,\overline{1}\right);\left(1,\overline{1}\right);\left(-1,\overline{1}\right)\right\}	
\]
admits no one-by-one exhaustion: Figure \ref{eigenfunction} shows an eigenfunction for the combinatorial Laplacian of this graph which takes exactly three values which are $-1,0,1$ and whose  support consists in the two vertices of $\Omega_0\subset\Omega_1$. The eigenvalue equals $6/5$. 

There are  elementary sufficient combinatorial conditions on a graph for the existence of eigenfunctions with finite support. Here is an example.
\begin{exam}
If a graph $\Gamma$ is obtained from a graph $H$ by adding to it a square formed  with four vertices $w,s,e,n$ and four edges $w-s$, $s-e$, $e-n$, $n-w$, and by gluing the vertex $e$ of the square, to a chosen vertex $u$ of the graph $H$, and by gluing the vertice $w$ of the square, to a chosen vertex $v$ of the graph $H$, then the function $\varphi$ which vanishes on all the vertices of $\Gamma$, except on $n$ and $s$ where it takes the values $\varphi(n)=1$ and $\varphi(s)=-1$, is an eigenfunction for the adjacency operator on $\Gamma$ of eigenvalue $0$.
\end{exam}
The existence of an height function, in the sense of Theorem \ref{theorem: exhaustion by level sets}, is a sufficient condition on a graph for the absence of eigenfunctions with finite support.

   The hypothesis of Theorem \ref{theorem: indicable}  imply the existence of an height function (see Theorem \ref{theorem: exhaustion by level sets}) which in turn implies the existence of a one-by-one exhaustion and  of the $\lambda$-unicity for any finite set of vertices and any $\lambda\in\mathbb R$. 

The second step in the proof of Theorem \ref{theorem: indicable} is an application of the localization principle for eigenfunctions: the existence of a  square summable $\lambda$-eigenfunction implies the existence of a  finitely supported $\lambda$-eigenfunction. A proof of this implication in our setting follows from the equivalence between conditions (4) and (5) in Theorem \ref{theorem: general principles}. The equivalence is proved with the help of the integrated density of states of $\Delta$ (see Subsection \ref{subsection: integrated density of states}) and a F{\o}lner sequence (see Subsection \ref{subsection: bound jumps}). The idea we follow
goes back to a short note published in 1984 by the physicists Delyon and Souillard \cite{DelSou}. See also the work of Bellissard: \cite[Proposition 4.1.4]{BelGap} and its proof and \cite[Question 2 page 116]{Bel} about the continuity of the integrated density of states, as well as Shubin's formula \cite[Appendix pages 146-148]{Bel} for computing the von Neumann trace with the help of usual traces and a F{\o}lner sequence.  Of course we do not claim any novelty about the equivalence of (4) and (5); our aim here is only to provide the reader with an easy and complete proof of the precise statement needed in our setting.  

The localization principle is a well-known fact studied and applied by many mathematicians and physicists in different contexts. 
It  seems  that  the  first  who  observed   the localization  principle  in  presence  of  a large  symmetry  was  Kuchment,  who  proved  it   in  the  case  of    lattices  
$\mathbb Z^d$, see \cite[Theorem 3]{Kuch} and the 1982 paper \cite[Theorem 12]{Kuch1982}. Later,  various  methods  were  applied  to   generalize  this  fact  to  larger  classes  of  groups,   eventually reaching  the  class  of  amenable  groups, that   perhaps  is  the largest  class of  groups  for  which  the  localization  principle  holds.  The  most  elegant  argument  given  in the  amenable  case    is  due  to  Elek  \cite[Proposition page 237]{Elek}  whose  note  is  based on  the  use of  von Neumann  dimension and  $L^2$-invariants  and  follows  the  ideas  of  Cheeger and Gromov  \cite[Lemma 3.1, equations 3.3 and 3.7]{CheeGrom}, Eckmann \cite[Theorem 1.2 and 3. $l_2$-cohomology, page 388]{Eck},   Dodziuk and  Matai  \cite[Theorem 0.1]{DodMat}  (a  lot  on  this  can  be  find  in Luck's  book \cite{Luck}).    The  localization  principle  for  amenable  groups  of  symmetries  is presented  in the  paper of Veselik \cite[Proposition 5.2]{Ves}  and  in  several  later  sources like for example in the paper of Higuchi and Nomura \cite[Theorem 3.2]{HiguchiNomura}.

\begin{figure}
    \begin{tikzpicture}
    \coordinate [label={above :$w_1$}] (1) at (2, 1);
    \coordinate [label={above :$v_1$}] (0) at (3, 1);
    
    \draw[dashed] (-3,1) -- (-2,1);
    \draw (-2,1) -- (4,1);
    \draw[dashed] (4,1) -- (5,1);
       
    \draw [dashed](-3,0) -- (-2,0);
    \draw (-2,0) -- (4,0);
    \draw[dashed] (4,0) -- (5,0);
    
    \draw (-2,0) -- (-2,1);
    \draw (-1,0) -- (-1,1);
    \draw (-0,0) -- (-0,1);
    \draw (1,0) -- (1,1);
    \draw (2,0) -- (2,1);
    \draw (3,0) -- (3,1);
    \draw (4,0) -- (4,1);
    
    \draw[dashed](-0.5,1.5)-- (2.5,1.5);
    \draw[dashed](-0.5,-0.5)-- (2.5,-0.5);
    \draw[dashed](-0.5,-0.5)-- (-0.5,1.5);
    \draw[dashed](2.5,-0.5)-- (2.5,1.5);
\end{tikzpicture}
    \caption{A one-by-one exhaustion (step 1).}
	\label{one}
\end{figure}
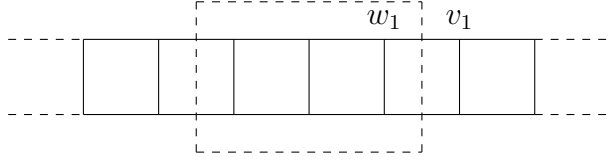

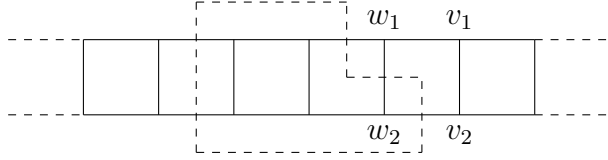
\begin{figure}
    \begin{tikzpicture}
    \coordinate [label={above :$w_1$}] (1) at (2, 1);
    \coordinate [label={above :$v_1$}] (0) at (3, 1);
    \coordinate [label={below :$w_2$}] (1) at (2, 0);
    \coordinate [label={below :$v_2$}] (0) at (3, 0);

    \draw[dashed] (-3,1) -- (-2,1);
    \draw (-2,1) -- (4,1);
    \draw[dashed] (4,1) -- (5,1);
       
    \draw [dashed](-3,0) -- (-2,0);
    \draw (-2,0) -- (4,0);
    \draw[dashed] (4,0) -- (5,0);
    
    \draw (-2,0) -- (-2,1);
    \draw (-1,0) -- (-1,1);
    \draw (-0,0) -- (-0,1);
    \draw (1,0) -- (1,1);
    \draw (2,0) -- (2,1);
    \draw (3,0) -- (3,1);
    \draw (4,0) -- (4,1);
    
    \draw[dashed](-0.5,1.5)-- (1.5,1.5);
    \draw[dashed](-0.5,-0.5)-- (2.5,-0.5);
    \draw[dashed](-0.5,-0.5)-- (-0.5,1.5);
    \draw[dashed](2.5,-0.5)-- (2.5,0.5);
    \draw[dashed](1.5,0.5)-- (2.5,0.5);
    \draw[dashed](1.5,0.5)-- (1.5,1.5);
\end{tikzpicture}
    \caption{A one-by-one exhaustion (step 2).}
	\label{two}
\end{figure}

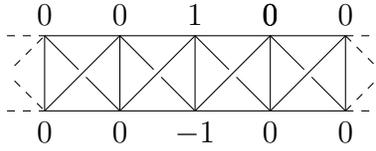
\begin{figure}
    \begin{tikzpicture}
    
    \draw [dashed] (-1.4,0.6) -- (-1,1);
    \draw [dashed] (-1.4,0.4) -- (-1,0);
    \draw (-1,0) -- (0,1);
    \draw (-1,1) -- (-0.55,0.55);
    \draw (-0.45,0.45) -- (0,0);
    \draw (0,0) -- (1,1);
    \draw (0,1) -- (0.45,0.55);
    \draw (0.55,0.45) -- (1,0);
    \draw (1,0) -- (2,1);
    \draw (1,1) -- (1.45,0.55);
    \draw (1.55,0.45) -- (2,0);
    \draw (2,0) -- (3,1);
    \draw (2,1) -- (2.45,0.55);
    \draw (2.55,0.45) -- (3,0);
    \draw [dashed] (3,0) -- (3.5,0);
    \draw [dashed] (3,0) -- (3.4,0.4);
    \draw [dashed] (3,1) -- (3.4,0.6);
    \draw [dashed] (3,1) -- (3.5,1);
    \draw[dashed] (-1.5,1) -- (-1,1);
    \draw (-1,1) -- (3,1);
       
    \draw [dashed](-1.5,0) -- (-1,0);
    \draw (-1,0) -- (3,0);

   \coordinate [label={above :$0$}] (1) at (-1, 1);
   \coordinate [label={above :$0$}] (1) at (0, 1);
   \coordinate [label={above :$1$}] (1) at (1, 1);
   \coordinate [label={below :$-1$}] (1) at (1, 0);
   \coordinate [label={below :$0$}] (1) at (2, 0);
   \coordinate [label={below :$0$}] (1) at (3, 0);
   \coordinate [label={above :$0$}] (1) at (2, 1);
   \coordinate [label={above :$0$}] (1) at (2, 1);
   \coordinate [label={below :$0$}] (1) at (0, 0);
   \coordinate [label={below :$0$}] (1) at (-1, 0);

    \coordinate [label={above :$0$}] (1) at (2, 1);
    \coordinate [label={above :$0$}] (0) at (3, 1);
    
    \draw (-1,0) -- (-1,1);
    \draw (-0,0) -- (-0,1);
    \draw (1,0) -- (1,1);
    \draw (2,0) -- (2,1);
    \draw (3,0) -- (3,1);

    \end{tikzpicture}
    \caption{An eigenfunction with finite support.}
	\label{eigenfunction}
\end{figure}

\subsection{The integrated density of states of $\Delta$}\label{subsection: integrated density of states}
We consider the spectral resolution $(E_{\lambda})_{\lambda\in\mathbb R}$ of the Laplace operator $\Delta$ of a weighted graph $\Gamma$ (see Subsection  \ref{subsection: spectral projections}). 
In the case $\Gamma$ admits a group of automorphisms with finitely many orbits of vertices 
(see Subsection \ref{subsection: the von Neumann trace}) - no other hypothesis about the action is needed - we may use  the von Neumann trace $\tau_1$ (see Definition \ref{definition: von Neumann trace}) to define the integrated density of states of $\Delta$ 
\[
	\mathbb R\to[0,1]
\]
\[
	\lambda\mapsto N(\lambda)=\tau_1(E_{\lambda}).
\]
The function $N$ is a non-decreasing right-continuous function which has a jump at $\lambda$ if and only if $\lambda$ is an eigenvalue of $\Delta$ (see Subsection \ref{subsection: spectral projections}). Lebesgue's theorem for the differentiability of monotone functions \cite[Ch. 1, no. 2]{RieszNagy} implies that $N$ is a.e. differentiable and  the Darboux-Froda's theorem implies that its set of points of discontinuity is at most countable.
The exact computation of $N$ is  probably hopeless for most graphs. The easiest case is when $\Gamma$ is the Cayley graph  of the infinite cyclic group $\mathbb Z$ with respect to the generating set $S=\{1\}$. In this case, the integrated density of states of the combinatorial Laplacian is expressed with the help of the level sets of the function
\[
	g:[-\pi,\pi]\to[0,2]
\]
\[
	g(\theta)=1-\cos\theta.
\]
More precisely, for each $\lambda\in\mathbb R$, we write
\[
	\{g\leq\lambda\}=\{\theta: g(\theta)\leq\lambda\}
\]
and denote ${\bf 1}_{\{g\leq\lambda\}}$ the characteristic function of this set.
One finds:
\begin{equation}\label{equation: integrated density of states}
\tau_1(E_{\lambda})=\frac{1}{2\pi}\int_{-\pi}^{\pi}{\bf 1}_{\{g\leq\lambda\}}(\theta)d\theta.
\end{equation}
This function is obviously analytic in $\lambda$. Analogous expressions are well-known in any dimension (i.e. on $\mathbb Z^d$, the free abelian group of rank $d\in\mathbb N$), see \cite[Example 1.4, Example 2.6, Formula (2.47), Example 9.6, Example 9.7]{Luck}.  

What about non-abelian cases? 
According to Corollary \ref{cor: lamplighter} above, the integrated density of states of the combinatorial Laplacian of the Cayley graph of the lamplighter group with respect to the generating set $\{a;c\}$ is continuous, whereas according to \cite{GriZuk}, if we consider the generating set $\{a;b\}$ instead, then the corresponding integrated density of states  restricted to $[0,2]$ has as a dense set of points of discontinuity. Even though this example shows that the regularity of the  integrated density of states is sensitive to the choice of the generating set, it is known \cite[Theorem 1.1]{BenPitSau} that its asymptotic behavior  near $0$ (more precisely its dilatational equivalence class near zero) is an invariant of the group (and more generally of its quasi-isometry class). In the case of the lamplighter group the dilatational equivalence class of $N(\lambda)$ is represented by the function
\[
	\lambda\mapsto e^{-\frac{1}{\sqrt\lambda}}.
\] 
A general formula \cite[Theorem 1.2]{BenPitSau} relating the asymptotic behavior of $N$ near zero to the asymptotic behavior of the $l^2$-isoperimetric profile near infinity brings estimates of $N$ for several families of finitely generated amenable groups \cite[1.7 Explicit computations]{BenPitSau}. (Good estimates for the $l^2$-isoperimetric profile can be obtained with the help of optimal F{\o}lner sequences.)  It follows from the technics explained in \cite{BenPitSau} that for any group $G$, the zero-dimensional Novikov-Schubin invariant (as defined in \cite[Definition 1.8]{LotLuc}) of a positive self-adjoint element $\Delta$ of the group algebra 
$\mathbb R[G]$ of the form $\Delta=e-\sum_ga_gg$, where $e\in G$ is the neutral element, with $a_g=a_{g^{-1}}\geq 0$ and $\sum_ga_g=1$ is bounded bellow by $1/2$ (or, by convention, is equal to the symbol $\infty^+$ if the support of $\sum_ga_gg$ generates a finite group). The lowest possible value $1/2$ is reached
by the combinatorial Laplacian of the Cayley graph of the infinite cyclic group $\mathbb Z$ with respect to the generating set $\{1\}$. The value $1/2$ characterizes (finite extensions of) the infinite cyclic group among finitely generated groups. See \cite[Theorem 1.2]{BenPitSau}.

\subsection{Large-scale geometry to bound jumps in the IDS}\label{subsection: bound jumps}
The Borel functional calculus applied to the Laplace operator $\Delta$ of a connected weighted graph $\Gamma$ associates to any $\lambda\in\mathbb R$ the orthogonal projection $E_{\{\lambda\}}$ onto the subspace of $\lambda$-eigenfunctions
(see Subsection \ref{subsection: spectral projections}).
In the case $\Gamma$ admits a cocompact group $G$ of automorphisms, we may use the von Neumann trace $\tau_1(E_{\{\lambda\}})$ to measure the von Neumann dimension of the space of $\lambda$-eigenfunctions (see Proposition \ref{proposition: properties of the von Neumann trace}). Whenever we use a von Neumann trace, in particular whenever we consider an integrated density of states, it is implicitly assumed that $\Gamma$ admits a cocompact group of automorphisms.  The projection $E_{\{\lambda\}}$ vanishes  if and only if the integrated density of states $N$ of $\Delta$ is continuous at $\lambda$ (see Subsection \ref{subsection: spectral projections}). If $\Delta$ satisfies $\lambda$-uniqueness on a finite subset $\Omega$ of the set $V$ of vertices of $\Gamma$, and if $\Omega$ is a disjoint union of fundamental domains for the action of $G$ on $V$, then
\[
\tau_1(E_{\{\lambda\}})\leq\frac{|\partial_2\Omega|}{|\Omega|},	
\]
where $|\partial_2\Omega|$ is the cardinality of the $2$-boundary of $\Omega$ (that is the set of vertices of $\Omega$ at distances less or equal to $2$ from $V\setminus\Omega$, see Definition \ref{definition: thick boundary}), and where $|\Omega|$ is the cardinality of $\Omega$ (see Theorem \ref{theorem: boundaries to bound von Neuman traces} below for the general case of a finite set $\Omega$, i.e. which is not necessarily partitioned by fundamental domains). A generalization of the above inequality is true in the setting of
Theorem \ref{theorem: boundaries to bound von Neuman traces}. When $\Gamma$ has a F{\o}lner sequence (see Definition \ref{definition: amenable graph}) and each set of the F{\o}lner sequence admits a one-by-one exhaustion, then the above inequality is a tool to prove the vanishing of all the 
projections $E_{\{\lambda\}}$, $\lambda\in\mathbb R$, or, in other words, a tool to prove the continuity of $N$. If there is no F{\o}lner sequence, the inequality  still provides an upper bound (always a bad one?) for the von Neumann dimensions of the spaces of eigenfunctions. In the presence of a F{\o}lner sequence, the inequality leads to a proof of the localization principle: the existence of a  $\lambda$-eigenfunction implies that the above inequality fails for at least one set, say $\Omega$, from the F{\o}lner sequence. Hence $\lambda$-unicity on $\Omega$ has to fail. It means that there is a   $\lambda$-eigenfunction whose support is included in $\Omega$. Theorem \ref{theorem: general principles} formalizes these ideas in a general setting (no hypothesis on the structure of the involved groups are needed; neither on the acting group  nor on the stabilizer subgroups). It generalizes
\cite[Theorem 3.2]{HiguchiNomura}. Similar ideas involving random Schr\"odinger operators are presented in \cite{Ves} and references therein. 

\subsection{Perspectives and questions}
Apply one-by-one exhaustions to show continuity of the integrated density of spaces for more operators, e.g. random Schr\"odinger operators.
In the presence of a one-by-one exhaustion, when is the spectral measure
absolutely continuous, when is it singular continuous? 
Compute von Neumann dimensions $\tau_1(E_{\{\lambda\}})$ for family of graphs, e.g.  planar Cayley graphs of crystallographic groups, Cayley graphs of $F_r\times\mathbb Z/2\mathbb Z$, the direct product of the (non-abelian if $r\geq 2$) free group of rank $r$ with the group of cardinality $2$, relative  to generating sets of the kind of the one involved in Figure \ref{eigenfunction}. 

\subsection{Acknowledgments}
We are grateful to Cosmas Kravaris for helping us to gain a better understanding of $\lambda$-uniqueness through examples. Rostislav Grigorchuk is grateful to Jean Bellissard and Peter Kuchment for numerous discussions about periodic graphs and integrated densities of states. Christophe Pittet is grateful to Jean Bellissard,  Alexander Bendikov and Roman Sauer for sharing their knowledge about integrated densities of states, spectral measures and $l^2$-invariants. Both authors are very grateful to Pierre de la Harpe for numerous corrections and suggestions improving the quality of the exposition. 

\section{Schr\"odinger operators on graphs}

\subsection{The path-metric on a connected graph and Cayley graphs}\label{subsection: the path-metric on a connected graph}
Recall that according to Serre \cite[2.1]{Ser},  
a \emph{graph} $\Gamma=(V,E,o,t,\iota)$ consists in: a set $V$ of \emph{vertices}, a set $E$ of \emph{oriented edges}, a map
		\[
			E\to V\times V
		\]
		\[
			e\mapsto (o(e),t(e)),
		\]
defining the \emph{origin}	$o(e)$ and the	\emph{terminus} $t(e)$ of the oriented edge $e$, and finally, a fixed point free involution $\iota$ defined on the set $E$,
with compatibility conditions:	
\[
o(\iota(e))=t(e),\, t(\iota(e))=o(e),\,\forall e\in E.	
\]
The involution ``flips the orientation of each oriented edge'' and is usually written as $\iota(e)=\overline{e}$.
(Notice that the above definition allows loops, multiple edges, and does not imply local finiteness.)

Let $x$ and $y$ be two vertices of $\Gamma$. A \emph{path of $\Gamma$ of finite length $n\in \mathbb N\cup\{0\}$ with origin $x$ and terminus $y$} is a sequence of vertices $v_0,\dots,v_n$
of $\Gamma$, such that $v_0=x$ and $v_n=y$, together with a sequence of $n$ edges
	$e_k,\, 1\leq k\leq n$
of $\Gamma$, such that $v_{k-1}=o(e_{k}), v_{k}=t(e_{k}),\,\forall 1\leq k\leq n.$
The graph $\Gamma$ is \emph{connected} if any two vertices of $\Gamma$ are the origin and terminus of a path of $\Gamma$ of finite length.
If $\Gamma$ is connected, the function $d:V\times V\to  \mathbb N\cup\{0\}$, 
$(x,y)\mapsto d(x,y),$
defined as the minimum  of the lengths of  the paths of $\Gamma$ with origin $x$ and terminus $y$, is a distance on $V$, also called the \emph{path-metric}  associated to $\Gamma$.

Let $G$ be a group and let $S\subset G$ be a generating set of $G$ (i.e. the only subgroup of $G$ containing $S$ is $G$ itself). The \emph{Cayley graph} 
$\Gamma=\mathcal C(G,S)$
 of $G$ relative to (the non-necessary symmetric) subset $S$ is defined as the graph with vertex set $V=G$, with edge set the disjoint union $$E=\left(G\times S\right)\bigsqcup \overline{\left(G\times S\right)},$$ 
and origin and terminus maps
$$(o((g,s)),t((g,s)))=(g,gs)$$ and hence $$\left(o\left(\overline{(g,s)}\right),t\left(\overline{(g,s)}\right)\right)=(gs,g).$$
Notice that $\mathcal C(G,S)$ is connected because $S$ is a generating set (symmetric or not). 
Let $\Gamma$ be a connected graph with path-metric $d$ defined on its vertex set $V$. For any $x\in V$ and $r\geq 0$, we define the \emph{closed ball of radius $r$ with center $x$} as
	\[
		B(x,r)=\{y\in V: d(x,y)\leq r\}.
	\]

\subsection{The Hilbert spaces associated to a weighted graph}\label{subsection: the Hilbert spaces associated to a weighted graph}
We consider  a strictly positive weight on the vertex set  $m_V: V\to]0,\infty[,$
and the associated Hilbert space
\[
	l^2(V,m_V)=\{\varphi:V\to\mathbb C: \sum_x|\varphi(x)|^2m_V(x)<\infty\}
\]
of square summable functions on $V$. The hermitian product of  $\varphi,\psi\in l^2(V,m_V)$ is:
\[
	\langle\varphi,\psi\rangle =\sum_x\varphi(x)\overline{\psi(x)}m_V(x),
\]
and $\|\varphi\|_2=\sqrt{\langle\varphi,\varphi\rangle}$ 
denotes the $l^2$-norm of $\varphi$.
We consider a strictly positive  weight on the oriented edge set  
$m_E: E\to]0,\infty[,$
which is \emph{symmetric} in the sense that
$m_E(e)=m_E(\overline{e}),\,\forall e\in E.$
A function $\alpha:E\to\mathbb C$ is \emph{anti-symmetric} if $\alpha(\overline{e})=-\alpha(e), \,\forall e\in E$.
Let 
\[
	l^2(E,m_E)=\{\alpha:E\to\mathbb C: \forall e\in E, \alpha(\overline{e})=-\alpha(e), \, \sum_{e\in E}|\alpha(e)|^2m(e)<\infty\}
\]
denote the  Hilbert space of anti-symmetric square summable functions on $E$. (We may thing of $l^2(E,m_E)$ as a space of $1$-forms on $\Gamma$.)
By definition, the hermitian product of $\alpha,\beta\in l^2(E,m_E)$ is:
\[
	\langle\alpha,\beta\rangle =\frac{1}{2}\sum_e\alpha(e)\overline{\beta(e)}m_E(e).
\]
The  $l^2$-norm of $\alpha\in l^2(E,m_E)$ is 
$\|\alpha\|_2=\sqrt{\langle\alpha,\alpha\rangle}.$ 

\subsection{The Laplace operator of a weighted graph}\label{Laplace}
For each vertex $x\in V$, we denote 
\[
	E_x=\{e\in E: o(e)=x\}
\]
the set of edges of $\Gamma$ whose origin is $x$.
\begin{defi}(The Sunada-Sy  necessary and sufficient condition for the boundedness of the Laplace operator \cite{SunSy}.)\label{definition: SuSy}
	 We say that a weighted graph satisfies the \emph{Sunada-Sy condition} if
\begin{equation*}\label{equation: bounded}
	\sup_{x\in V}\frac{1}{m_V(x)}\sum_{e\in E_x}m_E(e)<\infty.
\end{equation*}
\end{defi}

In this work, when we consider a weighted graph, we  always assume it satisfies the Sunada-Sy condition. Notice that the Sunada-Sy condition implies that $E_x$ is at most countable for any vertex $x$.
Let $f$ be a function on $V$. Assuming  the  Sunada-Sy condition, it is easy to check that  the formulae
\[
	d\varphi(e)=\varphi(t(e))-\varphi(o(e)),\,\forall e\in E
\]
define a bounded operator $d:l^2(V,m_V)\to l^2(E,m_E).$
Hence its adjoint $d^*:l^2(E,m_E)\to l^2(V,m_V),$
is also bounded. We conclude that the composition $d^*d:l^2(V,m_V)\to l^2(V,m_V),$
is bounded, self-adjoint, and positive. 
The \emph{Laplace operator} $\Delta$ associated to the weighted graph $(\Gamma,m_V,m_E)$
is defined as $\Delta=d^*d$.
The Sunada-Sy condition defined above is equivalent to the boundedness of $\Delta$ (see \cite{SunSy} for a proof). We will use the following well-known formulae for $\Delta$.

\begin{prop}\label{proposition: formula for the Laplacian}(The key formulae.) Let $\alpha\in l^2(E,m_E)$ and $\varphi\in l^2(V,m_V)$. Let $x\in V$.
\begin{enumerate}
	\item 
	\[
		d^*\alpha(x)=-\frac{1}{m_V(x)}\sum_{e\in E_x}\alpha(e)m_E(e).
	\]
	\item 
	\[
	\Delta \varphi(x)=\left(\frac{1}{m_V(x)}\sum_{e\in E_x}m_E(e)\right)\varphi(x)-\frac{1}{m_V(x)}\sum_{e\in E_x}\varphi(t(e))m_E(e).	
	\]
\end{enumerate}
\end{prop}

\subsection{Adjacency and  Markov operators, the combinatorial Laplacian}\label{subsection: adjacency and  Markov operators}
As explained in \cite{HiguchiNomura}, several familiar operators on $\Gamma$ may be seen as different avatars of the general Laplace
operator defined above.  
In the case $\sup_x|E_x|<\infty$, choosing $m_V$ and $m_E$ identically equal to $1$, we obtain
$\Delta=D-A$
where, for any $\varphi\in l^2(V,m_V)$ and $x\in V$, 
 $D\varphi(x)=|E_x|\varphi(x)$ 
is the diagonal multiplication operator by the degree at the vertex $x$ (i.e. the number of oriented edges with origin $x$) of $\Gamma$, and 
$$A\varphi(x)=\sum_{e\in E_x}\varphi(t(e))$$
is the \emph{adjacency operator}.
Suppose we are given a $w$-reversible random walk $p$ on $\Gamma$: that is a function
$p:E\to[0,1]$ satisfying 
$$\sum_{e\in E_x}p(e)=1,\, \forall x\in V,$$
	and a strictly positive function 
$w:V\to]0,\infty[$
such that 
\[
		w(o(e))p(e)=w(o(\overline{e}))p(\overline{e}), \forall e\in E.
\]
Choosing $m_V=w$ and defining the following weight on the edges:
	\[
		m_E(e)=w(o(e))p(e),\,\forall e\in E
	\]
we obtain $\Delta=I-M,$
where $I$ is the identity operator and $M$ is the \emph{Markov operator associated to $p$}, i.e. if $\varphi\in l^2(V,m_V)$ and $x\in V$,
	\[
		M\varphi(x)=\sum_{e\in E_x}\varphi(t(e))p(e).
	\]	
A special case of the previous setting is when $p$ is the \emph{simple random walk}, i.e.
	\[
		p(e)=\frac{1}{|E_{o(e)}|},\,\forall e\in E,
	\]
	and 
$w(x)=c|E_x|,\,\forall x\in V,$
where $c>0$ is any chosen constant.
For $\varphi\in l^2(V,m_V)$ and $x\in V$, we obtain the \emph{combinatorial Laplacian}
	\[
		\Delta \varphi(x)=\varphi(x)-\frac{1}{|E_x|}\sum_{e\in E_x}\varphi(t(e)).
	\]
In the case the degree $E_x$ does not depend on $x$, we may choose the constant $c=|{E_x}|^{-1}$ so that the weight $m_V=w$ is constant equal to $1$.

\subsection{Schr\"odinger operators}
\begin{defi}\label{definition: Schrodinger operator}(Schr\"odinger operator.) Let  $\Gamma$  be a  weighted graph satisfying the Sunada-Sy condition. Let $\Delta$ be its associated Laplacian.
Let $q:V\to \mathbb R$ be a bounded function (which will be called a \emph{potential}) on the vertex set $V$ of $\Gamma$:
\[
	\sup_{x\in V}|q(x)|<\infty.
\]
The \emph{Schr\"odinger operator} $H$ associated to $\Gamma$ and the \emph{potential} $q$, is the bounded self-adjoint operator
\[
	H=\Delta+q:l^2(V,m_V)\to l^2(V,m_V),
\]
i.e. if $\varphi\in l^2(V,m_V)$ and $x\in V$, then
\[
	H\varphi(x)=\Delta \varphi(x)+q(x)\varphi(x).
\]
\end{defi}
When working with an eigenfunction $\varphi$ of a Schr\"odinger operator $H$, it will be  convenient for us to apply an ``associated adjacency operator" $L$  to $\varphi$ and to control the result $L\varphi$.  The aim of the next proposition is to make this idea precise.
\begin{prop}\label{proposition: Schrodinger and Markov}(From Schr\"odinger to adjacency and back.)
Let  $\Gamma$  be a  weighted graph satisfying the Sunada-Sy condition. Let $V$ be the  vertex set of $\Gamma$. Let $q$ be a real bounded potential on $V$ and let $H$ be the corresponding Schr\"odinger operator on $l^2(V,m_V)$. Let $L$ be the bounded operator on $l^2(V,m_V)$, defined on each $\varphi\in l^2(V,m_V)$ and $x\in V$ as:
		\[
			L\varphi(x)=\frac{1}{m_V(x)}\sum_{e\in E_x}\varphi(t(e))m_E(e).
		\]
		For each real $\lambda$, consider the bounded potential $p_{\lambda}$ on $V$, defined on each $x\in V$ as:
		\[
			p_{\lambda}(x)=q(x)-\lambda+\frac{1}{m(x)}\sum_{e\in E_x}m_E(e).
		\]
	Let $\varphi\in l^2(V,m_V)$.  Then $H\varphi=\lambda\varphi,$
		if and only if $L\varphi=p_{\lambda}\varphi$.
		
\end{prop}
\begin{proof} We have the following equality between operators: $$H-\lambda=p_{\lambda}-L.$$
Hence $\mbox{Ker}(H-\lambda)=\mbox{Ker}(L-p_{\lambda})$.
\end{proof}

\section{Boundary conditions}
\subsection{Uniqueness of eigenfunctions}\label{subsection: Uniqueness of eigenfunctions}
\begin{defi}($\lambda$-uniqueness for $H$ on $\Omega$.)\label{definition: uniqueness}  Let  $\Gamma$ be a weighted graph satisfying the Sunada-Sy condition. Let $\Omega\subset V$ be a subset of vertices of $\Gamma$.
Let $H$ be a Schr\"odinger operator on $\Gamma$ and $\lambda$ be a real number. The operator $H$ satisfies \emph{$\lambda$-uniqueness on $\Omega$} if the only there is no $\lambda$-eigenfunction of $H$ which vanishes outside of $\Omega$. Formally: if $\varphi\in  l^2(V,m_V)$ is such that 
\[
	 H\varphi=\lambda\varphi,
\]
and if  $\varphi(x)=0$ for all $x\in V\setminus\Omega$, then $\varphi(x)=0$ for all $x\in V$.
\end{defi}
Notice that if $\Omega'\subset\Omega$ and if $H$ satisfies $\lambda$-uniqueness  on $\Omega$ then $H$ satisfies $\lambda$-uniqueness  on $\Omega'$. On the empty set, $H$ satisfies $\lambda$-uniqueness for any $\lambda$. 

Let $m_{\Omega}$ be the restriction of $m_V$ to $\Omega$. Let 
\[
l^2(\Omega,m_{\Omega})=\{\varphi:\Omega\to\mathbb C: \sum_{x\in\Omega}|\varphi(x)|^2m_{\Omega}(x)<\infty\}	
\]
be the Hilbert space of square summable functions on $\Omega$.
The inclusion $\Omega\subset V$ induces a natural linear isometric embedding
\[
	U_{\Omega}:l^2(\Omega,m_{\Omega})\to l^2(V,m_V)
\]
which is defined as ``the extension by zero outside of $\Omega$'', more formally: if $\varphi\in l^2(\Omega,m_{\Omega})$, then $U_{\Omega}\varphi(x)=\varphi(x)$ in the case 
$x\in\Omega$, and $U_{\Omega}\varphi(x)=0$ in the case 
$x\in V\setminus\Omega$. Let 
\[
	U_{\Omega}^*:l^2(V,m_V)\to l^2(\Omega,m_{\Omega})
\]
be the adjoint of $U_{\Omega}$. Using $\delta$-functions (i.e. characteristic functions of singletons)
for $x\in V$ and  $\omega\in\Omega$  we have:
\[
	\langle U^*_{\Omega}\delta_x,\delta_\omega|_{\Omega}\rangle_{l^2(\Omega,m_{\Omega})}=\langle\delta_x,U_{\Omega}\delta_\omega|_{\Omega}\rangle_{l^2(V,m_V)}.
\] 
Hence:
\begin{equation}\label{equation: U^*=0} 
x\in V\setminus\Omega \implies U^*_{\Omega}\delta_x=0,
\end{equation}
\begin{equation}\label{equation: U^*=1} 
x\in\Omega \implies U^*_{\Omega}\delta_x=\delta_x|_{\Omega}.
\end{equation}
Notice that the composition $P_{\Omega}=U_{\Omega}U_{\Omega}^*$ is the orthonormal projection onto the subspace
$U_{\Omega}(l^2(\Omega,m_{\Omega}))\subset l^2(V,m_V)$. In other words, for any $\varphi\in l^2(V,m_V)$, we have:
\[
	P_{\Omega}\varphi=\sum_{x\in\Omega}\varphi(x)\delta_x.
\]
\begin{defi}\label{definition: thick boundary}(Thick boundary.) Let $(X,d)$ be a metric space. For $E\subset X$ and  
	$x\in X$, we write
	\[
		d(x,E)=\inf_{y\in E}d(x,y).
	\]
For $r\geq 0$, the $r$-\emph{boundary} of a subset $\Omega\subset X$ is the subset of $X$ defined as:
	\[
		\partial_r\Omega=\{x\in \Omega: d(x,X\setminus\Omega)\leq r\}.
	\]
In other words $\partial_r\Omega$ consists in points  of $X$ lying inside $\Omega$ at depth less or equal to $r$. 
\end{defi}
This definition will mainly be applied to the vertex set  of a connected graph with its path-metric. In this case, all distances take integral values. The case $r=2$ is relevant for  uniqueness properties in the Dirichlet problem: roughly speaking we try to control $\lambda$-eigenfunctions on a domain $\Omega$ with the help of a condition on $\partial_2\Omega$. The following technical lemma will be useful. 

\begin{lemm}\label{lemma: gluing} (Cutting and pasting eigenfunctions.) Let $\Gamma$ be a weighted graph  with vertex set $V$. Assume $\Gamma$ is connected and satisfies the Sunada-Sy condition. Let $H=\Delta+q$ be a Schr\"odinger operator on $\Gamma$ whose potential $q$ is bounded.
Consider $\Omega\subset V$ and $\lambda\in\mathbb R$ and suppose $\varphi\in l^2(V,m_V)$ is a $\lambda$-eigenfunction of $H$:
\[
	H\varphi=\lambda\varphi.
\]
Assume $\varphi$ vanishes on $\partial_2\Omega$. Then:
\begin{enumerate}
	\item the function $P_{\Omega}\varphi\in l^2(V,m_V)$ either vanishes everywhere (i.e. is the zero function on $V$), or is also a $\lambda$-eigenfunction of $H$,
    \item and in the case $H$ satisfies $\lambda$-uniqueness on $\Omega$, the function $P_{\Omega}\varphi$ vanishes everywhere.
\end{enumerate} 
\end{lemm}

\begin{proof} Let $L$ and $p_{\lambda}$ be the bounded operator and the bounded potential defined in Proposition \ref{proposition: Schrodinger and Markov}.  We know that the equations
	\[
		H\varphi=\lambda\varphi,
	\]
	and
	\[
		L\varphi=p_{\lambda}\varphi,
	\]
are equivalent. Hence, in order to prove the first implication of the lemma, it is enough to show that the equality $L\varphi=p_{\lambda}\varphi$ implies 
\[
	LP_{\Omega}\varphi(x)=p_{\lambda}(x)P_{\Omega}\varphi(x),\,\forall x\in V.
\]
We consider two cases.
\begin{enumerate}
	\item Assume $d(x,V\setminus\Omega)\geq 2$ (in other words: $x$ lies in $\Omega$ at depth $2$ or more).
	When restricted to $\Omega$, the functions  $P_{\Omega}\varphi$ and $\varphi$ are equal. Hence, on one hand 
	we have: 
	\[
		P_{\Omega}\varphi(x)=\varphi(x),
	\]
	and on the other hand,
	\[
		LP_{\Omega}\varphi(x)=L\varphi(x)
	\]
	(because $t(e)\in\Omega$ for any $e\in E_x$).
	\item Assume $d(x,V\setminus\Omega)<2$ (in other words: either $x$ does not belong to $\Omega$ or $x$ lies at depth $1$ in $\Omega$).
	If $x$ does not belong to $\Omega$ then $P_{\Omega}\varphi(x)=0$. If $x$ lies at depth $1$ in $\Omega$, then $\varphi(x)=0$ by hypothesis  (notice, for later use - when proving below that $LP_{\Omega}\varphi(x)$ vanishes - that the same conclusion holds if $x$ lies at depth $2$)
	and, as already mentioned,  $P_{\Omega}\varphi$ and $\varphi$ are equal on $\Omega$. So again $P_{\Omega}\varphi(x)=0$. We claim that
	$LP_{\Omega}\varphi(x)=0$ too.  To prove this claim, notice first that if  $e\in E_x$, then $t(e)$  belongs either to the thick boundary of
	$\Omega$, or to $V\setminus\Omega$. In both cases $P_{\Omega}\varphi(t(e))=0$ as explained before the claim. This proves the claim. 
\end{enumerate}
The second implication in the lemma follows immediately from the first implication.
\end{proof}

\subsection{One-by-one exhaustions}
\begin{defi}\label{definition: exhaustion}(One-by-one exhaustion.) Let $\Gamma$ be a connected graph with vertex set $V$ and path-metric $d$. Let $\Omega\subset V$. A \emph{one-by-one exhaustion} $(V_n)_{n\geq 0}$ of $\Omega$
is a countable non decreasing sequence of subsets of $V$,
\[
	V_0\subset\cdots\subset V_n\subset V_{n+1}\subset\cdots
\]
satisfying the following conditions:
\begin{enumerate}
	\item $V_0=V\setminus\Omega$,\label{exhaustion start}
	\item for each $n$, there exists $x_n\in V_n$, such that
	\[
		 V_{n+1}=V_n\cup B(x_n,1),
	\]\label{exhaustion ball} 
	\item either  $|V_{n+1}\setminus V_n|=1$ or $V_n=V$,\label{exhaustion one-by-one }
	\item $\bigcup_nV_n=V$.\label{exhaustion end}
	
\end{enumerate}
\end{defi}

\begin{lemm}\label{lemma: one way into Omega}(One way into $\Omega$.)
Let $\Gamma$ be a connected graph with vertex set $V$ and path-metric $d$. Let $\Omega\subset V$. Let $(V_m)_{m\geq 0}$ be a one-by-one exhaustion of $\Omega$. Assume that for a given integer $n\geq 0$, the set  $V_n$ is strictly contained in $V$ (that is there exists at least one vertex of $\Gamma$ which is not in $V_n$). Let  $x_n\in V_n$, such that $ V_{n+1}=V_n\cup B(x_n,1)$.
Let $y_{n+1}$ be the unique element of $V_{n+1}\setminus V_n$.
Then
\[
	E_{x_n}=\{e\in E_{x_n}: t(e)=y_{n+1}\}\cup \{e\in E_{x_n}: t(e)\in V_n\}.
\]	
\end{lemm}
\begin{proof} The inclusion
\[
	E_{x_n}\supset\{e\in E_{x_n}: t(e)=y_{n+1}\}\cup \{e\in E_{x_n}: t(e)\in V_n\}	
\]	
is tautological. In order to prove the other inclusion, let $e\in E_{x_n}$. We have: 
\[
	t(e)\in B(x_n,1)\subset V_n\cup B(x_n,1)=V_{n+1}=V_n\cup\{y_{n+1}\}.
\]	
\end{proof}

\begin{theo}\label{theorem: exhaustion implies lambda-uniqueness}(One-by-one exhaustion implies $\lambda$-uniqueness.) Let $\Gamma$ be a weighted graph with vertex set $V$. Assume $\Gamma$ is connected and satisfies the Sunada-Sy condition. Let $H$ be a Schr\"odinger operator on $\Gamma$. If $\Omega\subset V$ admits a one-by-one exhaustion, then $H$ satisfies $\lambda$-uniqueness on $\Omega$ for any $\lambda\in\mathbb R$.
\end{theo}

We emphasize again that the above statement requires no hypothesis on the automorphism group of the graph (no group is involved here). The theorem has the following obvious corollary:

\begin{coro} Let $\Gamma$ be a weighted graph with infinite vertex set $V$. Assume $\Gamma$ is connected and satisfies the Sunada-Sy condition. Let $H$ be a Schr\"odinger operator on $\Gamma$. Assume for any finite set $F\subset V$ of vertices, there exists a finite set $\Omega\subset V$ which contains $F$ and which admits a one-by-one exhaustion. Then $H$ has no square summable eigenfunction with finite support.
\end{coro}

We prove the theorem.
\begin{proof} Let $\varphi\in l^2(V,m_V)$ such that $H\varphi=\lambda\varphi$. We have to prove the implication:
\[
	(\forall x\in V\setminus\Omega\,\,\, \varphi(x)=0)\implies  (\forall x\in V\,\,\, \varphi(x)=0).
\]
The implication is obviously true if $V\setminus\Omega$ is empty. Hence we may assume 
$V_0=V\setminus\Omega$ is non-empty. Let $(V_n)_{n\geq 0}$ be a one-by-one exhaustion of $\Omega$. We proceed by induction on $n$. By hypothesis, we know that the restriction of $\varphi$ to $V_0$ is identically equal to zero. Assume that $V_n$ is strictly included in $V$ and assume that the restriction of $\varphi$ to $V_n$ is identically equal to zero. Let us show that this implies that the restriction of $\varphi$ to $V_{n+1}$ is identically equal to zero. Notice that this will finish the proof because of Conditions \ref{exhaustion end} and \ref{exhaustion one-by-one } in Definition \ref{definition: exhaustion}.  Let $x_n\in V_n$ as in Condition \ref{exhaustion start} in Definition \ref{definition: exhaustion}. 
Our induction hypothesis, implies that
\[
	\varphi(x_n)=0.
\]
Let $L$ and $p_{\lambda}$ be defined as in Proposition \ref{proposition: Schrodinger and Markov}. Applying Proposition \ref{proposition: Schrodinger and Markov} to the hypothesis $H\varphi=\lambda\varphi$, we obtain:
\[
	0=\varphi(x_n)=p_{\lambda}(x_n)\varphi(x_n)=L\varphi(x_n).
\]
We may rewrite this equality as:
\[
	\sum_{e\in E_{x_n}}\varphi(t(e))m_E(e)=0.
\]
As $V_n$ is strictly included in $V$,  Lemma \ref{lemma: one way into Omega} applies and implies that 
\[
	E_{x_n}=\{e\in E_{x_n}: t(e)=y_{n+1}\}\cup \{e\in E_{x_n}: t(e)\in V_n\}.	
\] 
Hence:
\[
	\sum_{e\in E_{x_n}:\, t(e)=y_{n+1}}\varphi(t(e))m_E(e)=-\sum_{e\in E_{x_n}:\, t(e)\in V_n}\varphi(t(e))m_E(e).
\]
According to the induction hypothesis, the right-hand side vanishes. This proves that 
\[
\varphi(y_{n+1})\left(\sum_{e\in E_{x_n}:\, t(e)=y_{n+1}}m_E(e)\right)=0.	
\]
As $|\{e\in E_{x_n}: t(e)=y_{n+1}\}|\geq 1$ (because $d(x_n,y_{n+1})=1$), and as the weight $m_E$ is strictly positive on all edges, we deduce that $\varphi(y_{n+1})=0$. Hence $\varphi$ vanishes
on $V_{n+1}=V_n\cup\{y_{n+1}\}$.
\end{proof}

\begin{theo}\label{theorem: exhaustion by level sets}(Height functions bring $\lambda$-uniqueness and exhaustions.) Let $\Gamma$ be a connected graph with vertex set $V$ and path-metric $d$. Let $\mathbb Z$ be the set of integers with its usual metric. Assume there exists a function $h:V\to\mathbb Z$ with the following properties:
\begin{enumerate}[label=(\Alph*)]
	\item for any $x,y\in V$, $|h(x)-h(y)|\leq d(x,y)$,
	\item for each $b\in V$ there exists \emph{exactly} one element $c\in V$ such that $d(b,c)=1$ and such that $h(c)=h(b)+1$,
	\item for each $b\in V$ there exists \emph{at least} one vertex $a\in V$ such that $d(a,b)=1$ and such that $h(a)=h(b)-1$.
\end{enumerate}
Then the following holds true.
\begin{enumerate}
	\item Any finite subset $\Omega\subset V$ 
admits a one-by-one exhaustion.
    \item If $\Gamma$ admits weights satisfying the Sunada-Sy condition and if $H=\Delta+q$ is a Schr\"odinger operator on the weighted graph $\Gamma$, defined by a bounded potential $q$, then $H$ satisfies $\lambda$-uniqueness, for any $\lambda\in\mathbb R$, on any subset $\Omega\subset V$ such that
\[
	\min_{\omega\in\Omega}h(\omega)>-\infty.
\] 
\end{enumerate}
\end{theo}
\begin{exam} A Busemann function on a simplicial tree without leaves gives an example of a height function.
\end{exam}
\begin{exam} (An height function on a Baumslag-Solitar group.) Let $G$ be the group generated by the affine transformations of the real line $a(x)=2x$ and $b(x)=x+1$. There is an exact sequence of groups,
	\[
		\{0\}\to\mathbb Z[1/2]\to G\to\mathbb Z\to\{0\}
	\]
where the projection $h:G\to\mathbb Z$ satisfies $h(a)=1$ and $h(b)=0$. The homomorphism $h$ is an example of a height function on the Cayley graph of $G$ with respect to $S=\{a;b\}$. 
\end{exam}
\begin{proof} We first prove that any finite subset $\Omega\subset V$ 
admits a one-by-one exhaustion.  We proceed by induction on $k=|\Omega|$. If $k=0$ then $\Omega$ is empty and $V_0=V$ is a one-by-one exhaustion. Assume we are given $\Omega$ with $|\Omega|=k\geq 1$ and we know by induction hypothesis that any subset with strictly less than $k$ elements admits a one-by-one exhaustion. Let $V_0=V\setminus\Omega$. Let $y\in \Omega$ such that 
\[
h(y)=\min_{\omega\in\Omega}h(\omega).
\]
Let $V_1=V_0\cup\{y\}$. 
By hypothesis there exists $x\in V$ such that	
$d(x,y)=1$ and such that $h(x)=h(y)-1$.  Obviously, $x\in V_0$.
Let us check that
\[
	V_1=V_0\cup B(x,1).
\]
The inclusion $V_1\subset V_0\cup B(x,1)$ is obvious. In order to prove the other inclusion, it is enough to prove that if $z\in B(x,1)$,
then either $z\in V_0$ or $z=y$. We know that
\[
	|h(z)-h(x)|\leq d(z,x)\leq 1.
\]
Hence $h(z)=h(x)+\epsilon$ where $\epsilon\in\{-1;0;1\}$. But
\[
	h(z)=h(x)+\epsilon=h(y)-1+\epsilon.
\]
If $\epsilon\in\{-1;0\}$, then $z\in V_0$. If $\epsilon=1$, then $h(z)=h(y)$. In this case, both $y$ and $z$ are at distance exactly one from $x$, at equal height $h(x)+1$. This forces $z=y$.
By induction hypothesis, the set $\Omega\setminus\{y\}$ admits a one-by-one  exhaustion $W_0=V\setminus(\Omega\setminus\{y\})\subset\cdots\subset W_n$.
We have $V_0\subset V_1=W_0$. Setting $V_i=W_{i-1}$ for $1\leq i\leq n+1$ defines a one-by-one exhaustion of $\Omega$.

Now we prove that $H$ satisfies $\lambda$-uniqueness on any subset $\Omega\subset V$ such that
\[
	\min_{\omega\in\Omega}h(\omega)>-\infty.
\] 
Let $\lambda\in\mathbb R$ and let  $\varphi\in l^2(V,m_V)$. Assume  that $H\varphi=\lambda\varphi$. 
We have to prove the implication:
\[
	(\forall x\in V\setminus\Omega\,\,\, \varphi(x)=0)\implies  (\forall x\in V\,\,\, \varphi(x)=0).
\]
We proceed by contradiction: suppose $\varphi$ is not identically equal to zero. Then we may choose $\omega_0\in\Omega$ such that $\varphi(\omega_0)\neq 0$ and such that
\[
	h(\omega_0)=\min\{\ h(\omega): \omega\in\Omega, \varphi(\omega)\neq 0\}.
\]
By hypothesis, there exists $a\in V$ such that $d(a,\omega_0)=1$ and such that $h(a)=h(\omega_0)-1$. We have $\varphi(a)=0$. We proceed as in the proof of Theorem \ref{theorem: exhaustion implies lambda-uniqueness}. Namely, we apply Proposition \ref{proposition: Schrodinger and Markov} to obtain:
\[
	0=\varphi(a)=p_{\lambda}(a)\varphi(a)=L\varphi(a)=\sum_{e\in E_{a}}\varphi(t(e))m_E(e).\label{equation: 0=L}
\]
We claim that
\[
E_a=\{e\in E_{a}:\, t(e)=\omega_0\}\cup\{e\in E_{a}:\, t(e)\in V\setminus \Omega\}.	
\]
One  inclusion is obvious. In order to prove the other, let $e\in E_a$. We have:
\[
	|h(t(e))-a|\leq d(t(e),a)\leq 1.
\]
Hence there exists $\epsilon\in\{-1;0;1\}$ such that $h(t(e))=h(a)+\epsilon$. If $\epsilon\in\{-1;0\}$ then
\[
	h(t(e))\leq h(a)<h(\omega_0),
\]
hence $t(e)\in V\setminus\Omega$. In the case $\epsilon=1$,
\[
	h(t(e))= h(a)+1=h(\omega_0).
\]
Hence, as $t(e)$ and $\omega_0$ are two vertices at distance exactly $1$ from the vertex $a$, this forces $t(e)=\omega_0$. This finishes the proof of the claim. Applying the claim to Equation \ref{equation: 0=L}, we obtain:
\[
\sum_{e\in E_{a}:\, t(e)=\omega_0}\varphi(t(e))m_E(e)=-\sum_{e\in E_{a}:\, t(e)\in V\setminus\Omega}\varphi(t(e))m_E(e).
\]
The right-hand side vanishes by hypothesis. The left-hand side equals
\[
\varphi(\omega_0)\left(\sum_{e\in E_{a}:\, t(e)=\omega_0}m_E(e)\right).
\]
As the weights $m_E(e)$ are all strictly positive and as $$|\{e\in E_{a}:\, t(e)=\omega_0\}|\geq 1,$$ (because $d(a,\omega_0)=1$), we deduce that
$\varphi(\omega_0)=0$. This is a contradiction. Hence $\varphi$ has to be identically zero.
\end{proof}
	
\section{Dimensions of eigenspaces}
\subsection{Spectral projections}\label{subsection: spectral projections} 
Let $\mathcal H$ be a Hilbert space. Let $B(\mathcal H)$ be the $C^*$-algebra of bounded operators on $\mathcal H$. If $A\in B(\mathcal H)$, let $A^*$ denote its adjoint. Let
\[
	\|A\|=\sup_{\|v\|\leq 1}\|A(v)\|,
\]
be the operator norm of $A$.
Assume $A\in B(\mathcal H)$ is self-adjoint.
The spectrum $\sigma(A)$ of $A$ is a compact subset of the real line. Let $\mathcal B(\sigma(A))$ be the $C^*$-algebra of bounded Borel function on $\sigma(A)$, with involution defined by the equalities $f^*(x)=\overline{f(x)},\, \forall x\in\sigma(A)$,  and norm 
\[
	\|f\|=\sup_{x\in \sigma(A)}|f(x)|.
\]
For later use and for fixing the notation, we recall the Borel functional calculus form of the spectral theorem (see for example \cite[Theorem VII.2]{ReedSimon} or \cite[12.24]{Rudin}).
There is a unique map
\[
	\Phi:\mathcal B(\sigma(A))\to B(\mathcal H),
\]
with the following properties:
\begin{enumerate}
	\item $\Phi$ is $*$-morphism of algebras,
	\item $\forall f\in \mathcal B(\sigma(A)),\, \|\Phi(f)\|\leq\|f\|$,
	\item let $f$ be the function defined as $f(x)=x,\,\forall x\in\sigma(A)$, then $\Phi(f)=A$,
	\item if $f_n\in \mathcal B(\sigma(A))$, $n\in\mathbb N$, is a sequence which converges point-wise to $f\in \mathcal B(\sigma(A))$, then the sequence $\Phi(f_n)$, $n\in\mathbb N$, converges strongly to $\Phi(f)$.
\end{enumerate}

Let $\lambda\in\mathbb R$. Consider the Borel sets 
\[
]-\infty,\lambda]\cap\sigma(A),\,\,\,\,\{\lambda\}\cap\sigma(A),\,\,\,\,]-\infty,\lambda[ \cap \sigma(A), 
\]
their characteristic functions
\[
{\bf 1}_{]-\infty,\lambda]\cap\sigma(A)},\,\,\,\,{\bf 1}_{\{\lambda\}\cap\sigma(A)},\,\,\,\,{\bf 1}_{]-\infty,\lambda[\cap\sigma(A)},	
\]
and the corresponding operators defined by the Borel functional calculus:
\begin{align*}
	&E_{]-\infty,\lambda]}=\Phi\left({\bf 1}_{]-\infty,\lambda]\cap\sigma(A)}\right),\\
	&E_{\{\lambda\}}=\Phi\left({\bf 1}_{\{\lambda\}\cap\sigma(A)}\right),\\
	&E_{]-\infty,\lambda[}=\Phi\left({\bf 1}_{]-\infty,\lambda[\cap\sigma(A)}\right).	 
\end{align*}

It is customary to use the short notation:
\[
	E_{\lambda}=E_{]-\infty,\lambda]}.
\]
(To the interested reader, we recommend the elementary construction of $ E_{\lambda}$ and $E_{\{\lambda\}}$ explained in \cite[106. Fonctions d'une transformation sym\'etrique born\'ee]{RieszNagy}.) 
We recall the following well-known properties of $E_{\{\lambda\}}$ we will need.
Let $A\in B(\mathcal H)$ be self-adjoint. Let $\lambda\in\mathbb R$ and let $E_{\{\lambda\}}=\Phi\left({\bf 1}_{\{\lambda\}\cap\sigma(A)}\right)$.
	\begin{enumerate}
		\item The operator $E_{\{\lambda\}}$ is an orthogonal projection:
		\[
		E_{\{\lambda\}}=E_{\{\lambda\}}^*=E_{\{\lambda\}}^2.	
		\]
		\item  The operator $E_{\{\lambda\}}$ is positive:
		\[
			\forall v\in\mathcal H,\, \langle E_{\{\lambda\}}v,v\rangle\geq 0.
		\]
		\item Let $B\in B(\mathcal H)$. Assume that $AB=BA$. Then:
		\[
			BE_{\{\lambda\}}=E_{\{\lambda\}}B.
		\] 
	\end{enumerate} 
Recall also that an eigenspace is the range of a spectral projection. More precisely, let $A$ and $E_{\{\lambda\}}$ be as above. Then the image of the operator $E_{\{\lambda\}}$ is the $\lambda$-eigenspace of $A$:
	\[
	\emph{Im}\left(E_{\{\lambda\}}\right)=\{v\in \mathcal H: Av=\lambda v\}.
	\]
Recall also that the continuity of the integrated density is equivalent to the vanishing of spectral projections. More precisely,
consider any finite set $D\subset\mathcal H$ of vectors of $\mathcal H$ and the function
	\[
		N(\lambda)=\sum_{v\in D}\langle E_{\lambda}v,v\rangle.
	\]
The function $N(\lambda)$ is continuous at $\lambda_0$ if and only if 
$$\sum_{v\in D}\langle E_{\{\lambda_0\}}v,v\rangle=0.$$

\subsection{Boundaries to bound eigenspaces dimensions}
In a Dirichlet problem one seeks for a function which solves a specified equation on a given region and which takes prescribed values on the boundary of the region. 
Lemma \ref{lemma:  boundaries to bound eigenspaces dimensions} together with Proposition \ref{proposition: finite dimensional reduction} below are suitable formalizations of a pervading idea of harmonic analysis: in a Dirichlet problem the ``size'' of a $\lambda$-eigenspace is often controlled by the ``shape'' of the boundary. Recall from Subsection \ref{subsection: Uniqueness of eigenfunctions} above that $U_{\Omega}$ denotes ``the extension by zero outside of $\Omega$''.

\begin{lemma}\label{lemma:  boundaries to bound eigenspaces dimensions}(Boundaries to bound eigenspaces dimensions.) Let $\Gamma$ be a weighted connected graph satisfying the Sunada-Sy condition. Let $H=\Delta+q$ be a Schr\"odinger operator on $\Gamma$ defined by a real bounded potential $q$. Let $\Omega$ be a finite subset of vertices of $\Gamma$. Let $\lambda\in\mathbb R$. Suppose $H$ satisfies  $\lambda$-uniqueness on $\Omega$. Then the dimension of the image of the endomorphism $U_{\Omega}^* E_{\{\lambda\}}U_{\Omega}$ of the finite-dimensional space $l^2(\Omega,m_{\Omega})$, is bounded above by the cardinality of the thick boundary of $\Omega$:
	\[
		\emph{rank}\left(U_{\Omega}^* E_{\{\lambda\}}U_{\Omega}\right)\leq|\partial_2\Omega|.
	\]
\end{lemma}
\begin{proof} Consider a family of $n$ functions $g_1,\dots,g_n$, belonging to $$\mbox{Im}\left(U_{\Omega}^* E_{\{\lambda\}}U_{\Omega}\right).$$
Assume that 
	\[
		n>|\partial_2\Omega|.
	\]
The lemma will be proved if we show that the family $g_1,\dots,g_n$ is linearly dependent. As 
\[
\dim l^2(\partial_2\Omega,m_{\partial_2\Omega})=|\partial_2\Omega|<n,	
\]
there exists
$(\alpha_1,\dots,\alpha_n)\in \mathbb C^n\setminus\{0\}$, such that
\[
	\sum_{i=1}^n\alpha_i\left(g_i|_{\partial_2\Omega}\right)=0.
\]
Consider the corresponding linear combination:
\[
	g=\sum_{i=1}^n\alpha_ig_i\in \mbox{Im}\left(U_{\Omega}^* E_{\{\lambda\}}U_{\Omega}\right).
\]	
Let $V$ denotes the vertex set of $\Gamma$.
Let $f\in l^2(\Omega,m_{\Omega})$ such that $$g=U_{\Omega}^* E_{\{\lambda\}}U_{\Omega}f,$$ and let
$$h=E_{\{\lambda\}}U_{\Omega}f\in l^2(V,m_V).$$
Hence (see Subsection \ref{subsection: spectral projections}):
$Hh=\lambda h$.

Claim:
\[
	h|_{\partial_2\Omega}=0.
\]
In order to prove the claim, we expend $h$ with the help of the Hilbert basis 
$\delta_x,\, x\in V$, of $l^2(V,m_V)$. Namely:
\[
	h=\sum_{x\in V}h(x)\delta_x.
\]
Hence, according to Implications \ref{equation: U^*=0} and \ref{equation: U^*=1}, 
\[
U^*_{\Omega}h=\sum_{x\in V}h(x)U^*_{\Omega}\delta_x=\sum_{\omega\in\Omega}h(\omega)\delta_{\omega}|_{\Omega}.	
\]
In particular, for any $\omega\in\Omega$,
\[
	h(\omega)=(U^*_{\Omega}h)(\omega)=(U^*_{\Omega}E_{\{\lambda\}}U_{\Omega}f)(\omega)=g(\omega).
\]
In the case $\omega\in\partial_2\Omega$, remembering that $g|_{\partial_2\Omega}=0$, we obtain $h(\omega)=0$.
This finishes the proof of the claim. We can now finish the proof of the lemma. Indeed, according to the second statement of Lemma \ref{lemma: gluing} (which applies because of the claim, the fact that $Hh=\lambda h$, and the hypothesis of $\lambda$-uniqueness),
\[
	P_{\Omega}h=0.
\]
We then have:
\begin{align*}
	0&=P_{\Omega}h=P_{\Omega}E_{\{\lambda\}}U_{\Omega}f\\
	 &=U_{\Omega}U_{\Omega}^*E_{\{\lambda\}}U_{\Omega}f\\
	 &=U_{\Omega}g.
\end{align*}
As $U_{\Omega}$ is one-to-one we deduce that $g=0$. We conclude that
\[
	\sum_{i=1}^n\alpha_ig_i=g=0,
\]
is a non-trivial linear combination, showing that the family $g_1,\dots,g_n$ is linearly dependent.
\end{proof}
We will need the following basic result from linear algebra (for a proof, we refer the reader to \cite[3.6]{Simon}).
\begin{lemma}\label{lemma: trace} Let $\mathcal H$ be a finite dimensional Hilbert space. Let $A\in B(\mathcal H)$ be an endomorphism of $\mathcal H$ (viewed as an operator). Then the usual trace of $A$ is bounded by the rank of $A$ times its norm:
	\[
		|Tr(A)|\leq \mbox{rank}(A)\|A\|.
	\]
\end{lemma}

\begin{prop}\label{proposition: finite dimensional reduction}(Finite-dimensional reduction.) Let $\Gamma$ be a weighted connected graph satisfying the Sunada-Sy condition. Let $H=\Delta+q$ be a Schr\"odinger operator on $\Gamma$ defined by a real bounded potential $q$. Let $V$ be the vertex set of  $\Gamma$  and let
$\Omega\subset V$ be a finite subset. Let $\lambda\in\mathbb R$. Then
\[
	\sum_{x\in \Omega}\langle E_{\{\lambda\}}\delta_x,\delta_x\rangle\leq \emph{rank}(U_{\Omega}^*AU_{\Omega}).
\] 
\end{prop}
\begin{proof} Let $A\in B(l^2(V,m_V))$. For any $x\in V$, it is obvious that $U_{\Omega}(\delta_x|_{\Omega})=\delta_x$. Hence, Lemma \ref{lemma: trace} leads to the following upper bound:
	
\begin{align*}
\sum_{x\in \Omega}\langle A\delta_x,\delta_x\rangle&=\sum_{x\in \Omega}\langle AU_{\Omega}(\delta_x|_{\Omega}),U_{\Omega}(\delta_x|_{\Omega})\rangle\\
                          &=\sum_{x\in \Omega}\langle U_{\Omega}^*AU_{\Omega}(\delta_x|_{\Omega}),\delta_x|_{\Omega}\rangle\\
                          &=Tr( U_{\Omega}^*AU_{\Omega})\\
						  &\leq \mbox{\mbox{rank}}(U_{\Omega}^*AU_{\Omega})\|U_{\Omega}^*AU_{\Omega}\|.
\end{align*}
According to Subsection \ref{subsection: spectral projections},
\[
	\|E_{\{\lambda\}}\|=\|\Phi({\bf 1}_{\{\lambda\}})\|\leq\|{\bf 1}_{\{\lambda\}})\|=1,
\] 
hence we see that $\|U_{\Omega}^*E_{\{\lambda\}}U_{\Omega}\|\leq 1$. Choosing $A=E_{\{\lambda\}}$, we obtain
\[
	\sum_{x\in \Omega}\langle E_{\{\lambda\}}\delta_x,\delta_x\rangle\leq \mbox{rank}(U_{\Omega}^*E_{\{\lambda\}}U_{\Omega}).
\]

\end{proof}
\begin{prop}\label{proposition: boundaries to bound traces}(Boundaries to bound traces.) Let $\Gamma$ be a weighted connected graph satisfying the Sunada-Sy condition. Let $H=\Delta+q$ be a Schr\"odinger operator on $\Gamma$ defined by a real bounded potential $q$. Let $V$ be the vertex set of  $\Gamma$  and let
$\Omega\subset V$ be a finite subset. Let $\lambda\in\mathbb R$. Assume $H$ satisfies $\lambda$-uniqueness on $\Omega$. Then
\[
	\sum_{x\in \Omega}\langle E_{\{\lambda\}}\delta_x,\delta_x\rangle\leq|\partial_2\Omega|.
\] 
\end{prop}
\begin{proof} The inequality follows from the combination of Proposition \ref{proposition: finite dimensional reduction} with Lemma \ref{lemma:  boundaries to bound eigenspaces dimensions}.
	
\end{proof}

\section{Large-scale geometry}
\begin{defi}(Packing number relative to a family.) Let $X$ be a set. Suppose $\mathcal F$ is a given family of \emph{non-empty} subsets of $X$. For any subset $\Omega$ of $X$, we define the \emph{packing number of $\Omega$ relative to the family $\mathcal F$} as the maximal number of disjoint elements of $\mathcal F$ which are included in $\Omega$, with the convention that this number is zero if no element of $\mathcal F$ is included in $\Omega$:
\[
		    P(\Omega,\mathcal F)=0\,\, \mbox{\emph{if there is no $F\in \mathcal F$ such that $F\subset\Omega$}},
\]
\[
		    P(\Omega,\mathcal F)=\max\{n:\bigsqcup_{1\leq i\leq n} F_i\subset\Omega,\,F_i\in\mathcal F\} \,\, \mbox{\emph{otherwise}}.
\]
\end{defi}
Notice that in the case $\mathcal F$ contains each singleton of $X$, then $P(\Omega,\mathcal F)$ coincides with the cardinality of $\Omega$.

\begin{defi}(Nets and maximal nets.) Let $(X,d)$ be a metric space. Let $r\geq 0$. An \emph{$r$-net of $X$} is a subset $N(X,r)\subset X$ such that
\[
	\forall x,y\in N(X,r),\, (d(x,y)\leq r)\implies (x=y).
\] 
A \emph{maximal $r$-net of $X$} is an $r$-net of $X$ which is maximal with respect to the inclusion relation among all $r$-nets of $X$.
	
\end{defi}
Let $(X,d)$ be a metric space. Let $x\in X$ and let $r\geq 0$. Consider the closed ball in $X$ with center $x$ and radius $r$:
\[
	B(x,r)=\{y\in X: d(y,x)\leq r\}.
\]
We will need the following definition from \cite{CorHar}.
\begin{defi}(Locally finite spaces with uniform growth.) A metric space $(X,d)$ is \emph{locally finite with uniform growth}  if for each $r\geq 0$,
\[
	\emph{Vol}(r)=\sup_{x\in X}|B(x,r)|<\infty.
\]
In what follows we  shorten the terminology and say ``a space with uniform growth'' - omitting the ``locally finite''.
\end{defi}

\begin{lemma}\label{lemma: bounding a set with a net}(Bounding a set with a net.) Let $(X,d)$ be a metric space with uniform growth. Let $\Omega\subset X$ be a subset.
For $r\geq 0$, let $N(\Omega,r)$ be a $r$-net of $\Omega$. 
If $N(\Omega,r)$ is maximal then
\[
	|\Omega|\leq \emph{Vol}(r)\cdot|N(\Omega,r)|.
\]
\begin{proof} As $N(\Omega,r)$ is a maximal $r$-net of $\Omega$ there exists a map
	\[
		p:\Omega\to N(\Omega,r)
	\]
whose maximal displacement 	is bounded by $r$:
\[
	\forall x\in \Omega,\, d(x,p(x))\leq r.
\]
Hence: 
\[
\forall y\in N(\Omega,r),\, p^{-1}(\{y\})\subset B(y,r).	
\]
We conclude that $|p^{-1}(\{y\})|\leq \mbox{Vol}(r)$ and as 
\[
	\Omega\subset\bigsqcup_{y\in  N(\Omega,r)}p^{-1}(\{y\})
\]
we deduce that
\[
	|\Omega|\leq  \mbox{Vol}(r)\cdot|N(\Omega,r)|.
\]
\end{proof}
	
\end{lemma}

\begin{defi}(The $r$-interior of a subset.) Let $(X,d)$ be a metric space. Let $\Omega\subset X$ and $r\geq 0$. The $r$-interior of $\Omega$ is the set $I(\Omega,r)$ of points of $\Omega$ which lies at ``depth" strictly greater than $r$. More precisely
	\[
		I(\Omega,r)=\Omega\setminus\partial_r\Omega=\{x\in\Omega: d(x,X\setminus\Omega)>r\}.
	\]
	
\end{defi}

\begin{defi}(Inclusive radius.) Let $(X,d)$ be a metric space. Let $x\in X$. Let $\mathcal F$ be a family of \emph{non-empty} subsets of $X$. The \emph{inclusive radius $r(x,\mathcal F)$ at $x$ relative to the family} $\mathcal F$ is the minimal $r\geq 0$ such that there exists a set $F\in\mathcal F$  such that
	\[
		F\subset B(x,r).
	\]
\end{defi}

\begin{prop}\label{proposition: interior points as a lower bound for the packing number}(Interior points as a lower bound for the packing number.) Let $(X,d)$ be a metric space with uniform growth $\emph{Vol}$. Let $\mathcal F$ be a family of \emph{non-empty} subsets of $X$. Assume  $0\leq r<\infty$ is a uniform upper bound for the inclusive radii relative to  $\mathcal F$:
	\[
		\forall x\in X,\,\exists F_x\in\mathcal F: F_x\subset B(x,r).
	\]	
Then the number $|I(\Omega,r)|$ of $r$-interior points of $\Omega$ satisfies
\[
	|I(\Omega,r)|\leq \emph{Vol}(2r)\cdot P(\Omega,\mathcal F).	
\]
\end{prop}
\begin{proof}  Let $N(I(\Omega,r),2r)$ be a $2r$-net of $I(\Omega,r)$. As $r$ is a uniform upper bound for the inclusive radii relative to  $\mathcal F$, we may choose for each point $x\in N(I(\Omega,r),2r)$ a set
$F_x\in\mathcal F$ such that 
\[
	F_x\subset B(x,r).
\]
Hence, as $x\in I(\Omega,r)$, it is obvious that
\[
	F_x\subset\Omega.
\]
Let $x,y\in N(I(\Omega,r),2r)$. By definition of a $2r$-net, 
\[
	(x\neq y)\implies (B(x,r)\cap B(y,r)=\emptyset). 
\]
Therefore, we obtain a disjoint union of subsets of $\Omega$,
\[
	\bigsqcup_{x\in N(I(\Omega,r),2r)}F_x\subset\Omega,
\]
where each subset $F_x$ belongs to the family $\mathcal F$. This proves that the cardinal of the net $N(I(\Omega,r),2r)$ is a lower bound for the packing number of $\Omega$ relative to $\mathcal F$:
\[
	|N(I(\Omega,r),2r)|\leq P(\Omega,\mathcal F).
\]
If we choose  the net $N(I(\Omega,r),2r)$ maximal among $2r$-nets of $I(\Omega,r)$ then Lemma \ref{lemma: bounding a set with a net} implies that
\[
	|I(\Omega,r)|\leq \mbox{Vol}(2r)\cdot|N(I(\Omega,r),2r)|.
\]
Finally:
\[
|I(\Omega,r)|\leq \mbox{Vol}(2r)\cdot P(\Omega,\mathcal F).	
\]
\end{proof}
We need the following concept, called a $1$-geodesic space in \cite{CorHar}.
\begin{defi}\label{definition: geodesic metric with integral values}(Geodesic distance with integer values.) Let $(X,d)$ be a metric space. We say that the distance $d$ is \emph{geodesic with integer values} if the following conditions hold true:
\begin{enumerate}
	\item $\forall x,y\in X,\, d(x,y)\in\mathbb N\cup\{0\}$,
	\item for all $x,z\in X$, for all integer $k$ such that $0\leq k\leq d(x,z)$, there exists $y\in X$ such that $d(x,y)=k$ and
	$$d(x,z)=d(x,y)+d(y,z).$$
	\end{enumerate}
\end{defi}

\begin{lemma}\label{lemma: comparing boundaries}(Comparing boundaries.) Let $(X,d)$ be a metric space with uniform growth $\emph{Vol}$. Assume $d$ is geodesic with integer values.
Let $\Omega\subset X$. Let $r>1$. Then
\[
	|\partial_r\Omega|\leq \emph{Vol}(r-1)\cdot |\partial_1\Omega|.
\] 
\end{lemma}
\begin{proof} Let $x\in \partial_r\Omega$. By definition of $\partial_r\Omega$ the set 
	\[
		Z=\{z\in X\setminus\Omega: d(x,z)\leq r\}
	\]
is non-empty. As $d$ takes integer values, there exists $z_0\in Z$ such that  
\[
	d(x,z_0)=d(x,Z).
\]
Notice that $x\neq z_0$ because $x\in\Omega$ and $z_0\in X\setminus\Omega$. Hence $d(x,z_0)\geq 1$. As $d$ is geodesic with integral values, there exists $y\in X$ such that $d(y,z_0)=1$ and such that
\[
	d(x,z_0)=d(x,y)+d(y,z_0).
\] 
The minimality of $z_0$ implies that $y\in\Omega$. From the facts that $z_0\in X\setminus\Omega$ and $d(y,z_0)=1$, we conclude that $y\in\partial_1\Omega$. Notice also that 
\[
	d(x,y)=d(x,z_0)-1=d(x,Z)-1\leq r-1.
\]
Denoting $y=p(x)$ we see that we have constructed a map
$$p:\partial_r\Omega\to\partial_1\Omega$$
such that $d(x,p(x))\leq r-1$. As in the proof of Lemma \ref{lemma: bounding a set with a net},	we bound the cardinalities of the fibres of $p$ and obtain
	\[
		|\partial_r\Omega|\leq \mbox{Vol}(r-1)\cdot |\partial_1\Omega|.
	\]

\end{proof}

The next lemma brings a lower bound on the number of interior points in a finite set. It will be applied to some finite sets of vertices in a graph which have a ``a relatively small boundary''. The idea of a finite set with ``a relatively small boundary'' is formalized through the definition of a F{\o}lner sequence (see Definition \ref{definition: amenable graph} below).
\begin{lemma}\label{lemma: folner to bound below interior points}(F{\o}lner to bound below interior points.)
 Let $(X,d)$ be a metric space with uniform growth $\emph{Vol}$. Assume $d$ is geodesic with integer values.
Let $\Omega\subset X$. Let $r>1$. Assume   $\epsilon\geq 0$ is such that
\[
	\emph{Vol}(r)\cdot|\partial_1\Omega|\leq\epsilon|\Omega|.
\]	
Then the number of $r$-interior points is bounded below as follows:
\[
	|I(\Omega,r)|\geq(1-\epsilon)|\Omega|.
\]
\end{lemma}
\begin{proof} According to Lemma \ref{lemma: comparing boundaries},
	\[
		|\partial_r\Omega|\leq \mbox{Vol}(r)\cdot|\partial_1\Omega|.
	\]
Hence,
\begin{align*}
	|I(\Omega,r)|&=|\Omega\setminus\partial_r\Omega|=|\Omega|-|\partial_r\Omega|\\
				 &\geq |\Omega|-\mbox{Vol}(r)\cdot|\partial_1\Omega|\geq |\Omega|-\epsilon|\Omega|\\
				 &=(1-\epsilon)|\Omega|.
\end{align*}
\end{proof}

\section{Groups and quasi-homogeneous graphs}
\subsection{Groups acting on graphs}\label{subsection: group action}
Let $G$ be a group acting (on the left) on a graph $(\Gamma, V,E,o,t,\iota)$. It means that $G$ acts both on $V$ and on $E$ and that the two actions are compatible in the sense that
$$\forall g\in G, \forall e\in E,\, o(ge)=go(e),\, t(ge)=gt(e), \iota(ge)=g\iota(e).$$
For example, a group $G$ acts on any of its Cayley graph $\mathcal C(G,S)$ and the action preserves the path metric.
If the group $G$ acts on a weighted graph, we always assume that the weights are invariant:
\[
	\forall g\in G, \forall x\in V, m_V(gx)=m_V(x), \forall e\in E, m_E(ge)=m_E(e).
\]
The permutation representation associated to a  $m_V$-preserving $G$-action on $V$  is the unitary representation defined by the following conditions:
\[
	\pi:G\to U(l^2(V,m_V)),
\]
\[
	\forall g\in G,\forall f\in l^2(V,m_V),\forall x\in V,\, (\pi(g)f)(x)=f(g^{-1}x).
\]
If the weighted graph $\Gamma$ satisfies the Sunada-Sy condition, and if the Schr\"odinger operator $H=\Delta+q$ is defined with the help of a $G$-invariant potential $q$, then 
\[
	\forall g\in G,\, \pi(g)H=H\pi(g).
\]
Hence $H$ belongs to the \emph{commutant}
\[
	\pi(G)'=\{A\in B(l^2(V,m_V): \forall g\in G, \pi(g)A=A\pi(g)\}
\]
of $\pi$. It is straightforward to check that $\pi(G)'$ is an algebra, which is stable under conjugation, contains the identity, and which is closed with respect to the strong operator topology. (In other words, it is a von Neumann algebra.)

\subsection{The von Neumann trace}\label{subsection: the von Neumann trace}

\begin{defi} A graph $\Gamma$ is \emph{quasi-homogeneous} if there exists a group $G$ acting on $\Gamma$ (in the sense explained in Subsection \ref{subsection: group action} above) such that the action of $G$ on the vertex set of $\Gamma$ has a finite number of orbits. (We don't require freeness of the action.)	
\end{defi}
Here and in what follows, no hypothesis on the stabilizers of the action is needed.

\begin{defi}\label{definition: von Neumann trace}(The von Neumann trace of a positive operator.)
Let $\Gamma$ be a weighted graph and let $G$ be a group acting on $\Gamma$. We assume that the weights $m_V$ and $m_E$ are $G$-invariant. Assume the vertex set $V$ of $\Gamma$ decomposes as a finite union of $G$-orbits (in the case the degrees of the vertices of $\Gamma$ are finite, this is equivalent to assume that the action of $G$ is cocompact). Let $D$ be a fundamental domain for this action. In other words, the vertex set $V$ is the disjoint union of the orbits of the vertices of the fundamental domain:
\[
	V=\bigsqcup_{x\in D} Gx.
\]
Let $A\in B(l^2(V,m_V)$ be in the commutant $\pi(G)'$ of the permutation representation $\pi(G)$ as defined in Subsection \ref{subsection: group action} above. Assume that $A$ is a positive operator, that is:
\[
	\forall \varphi\in l^2(V,m_V),\, \langle A\varphi,\varphi\rangle\geq 0.
\]
We define 
\[
	\tau(A)=\sum_{x\in D}\langle A\delta_x,\delta_x\rangle.
\]
We normalize $\tau$ and define
\emph{the von Neumann trace of $A$} as 
\[
	\tau_1(A)=\frac{1}{|D|}\sum_{x\in D}\langle A\delta_x,\delta_x\rangle.
\]
\end{defi}
We will use the following well-known properties of the von Neumann trace.
\begin{prop}\label{proposition: properties of the von Neumann trace}(Some properties of the von Neumann trace.) With the notation as above, let $A\in \pi(G)'$ be a positive operator. The following properties are true.
	\begin{enumerate}
		\item Neither $\tau(A)$ nor $\tau_1(A)$  depends on the choice of the fundamental domain,
		\item $0\leq \tau_1(A)\leq \|A\|$,
		\item $\tau_1(A)=0$ if and only if $A=0$.
	\end{enumerate}
\end{prop}

\begin{theo}\label{theorem: boundaries to bound von Neuman traces}(Boundaries to bound von Neumann traces.) Let $\Gamma$ be a connected weighted graph satisfying the Sunada-Sy condition. Let $G$ be a group acting on $\Gamma$ (as defined in Subsection \ref{subsection: group action}). (We make no freeness hypothesis.) We assume the weigths $m_V$ and $m_E$ are $G$-invariant. Let $H=\Delta+q$ be a Schr\"odinger operator on $\Gamma$ defined by a real valuated $G$-invariant potential $q$. Suppose the action of $G$ on the vertex set $V$ of $\Gamma$ has  a finite number of orbits. Let $\mathcal F$ be the family of all fundamental domains for the action of $G$ on $V$. Let $\Omega\subset V$ be a finite subset. Assume the packing number $P(\Omega,\mathcal F)$ of $\Omega$ with respect to the family $\mathcal F$ is nonzero: in other words, $\Omega$ contains as a subset at least one fundamental domain. Let $\lambda\in\mathbb R$. Let $E_{\{\lambda\}}$ be the spectral projection onto the $\lambda$-eigenspace of $H$.  If $H$ satisfies $\lambda$-unicity on $\Omega$ then
	\[
		\tau \left(E_{\{\lambda\}}\right)\leq\frac{|\partial_2\Omega|}{P(\Omega,\mathcal F)}.
	\]
\end{theo}
\begin{proof} Let $n=P(\Omega,\mathcal F)$. Let $F_1,\dots,F_n$ be a collection of disjoint fundamental domains included in $\Omega$. As recalled in Subsection \ref{subsection: spectral projections}, the spectral projection  $E_{\{\lambda\}}$ is positive and belongs to the commutant $\pi(G)'$. Therefore, we may  apply Proposition \ref{proposition: properties of the von Neumann trace} to deduce that the von Neumann trace of $E_{\{\lambda\}}$ does not depend on the choice of the fundamental domain. We therefore have:
	\begin{align*}
		n\cdot\tau \left(E_{\{\lambda\}}\right)&=n\sum_{x\in F_1}\langle E_{\{\lambda\}}\delta_x,\delta_x\rangle\\
		                                  &=\sum_{x\in \bigsqcup_{i=1}^n F_i}\langle E_{\{\lambda\}}\delta_x,\delta_x\rangle\\
										  &\leq \sum_{x\in\Omega}\langle E_{\{\lambda\}}\delta_x,\delta_x\rangle.\\
\end{align*}
As $H$ satisfies $\lambda$-unicity, Proposition \ref{proposition: boundaries to bound traces} applies hence we deduce that
\[
	\sum_{x\in\Omega}\langle E_{\{\lambda\}}\delta_x,\delta_x\rangle\leq |\partial_2\Omega|.
\]	
\end{proof}

\subsection{Quasi-homogeneous graphs with a F{\o}lner sequence}
\begin{defi}\label{definition: amenable graph}(F{\o}lner sequence in a graph.) Let $\Gamma$ be a connected graph. A \emph{F{\o}lner} sequence in $\Gamma$ is a sequence $(\Omega_n)_{n\geq 0}$ of finite subsets of the vertex set of $\Gamma$ such that 
	\[
		\lim_{n\to\infty}\frac{|\partial_1\Omega_n|}{|\Omega_n|}=0.
	\]	
\end{defi}

\begin{theo}\label{theorem: general principles}(Continuity of the integrated density of states.) Let $\Gamma$ be a connected weighted graph satisfying the Sunada-Sy condition and admitting a F{\o}lner sequence. Let $G$ be a group acting on $\Gamma$ (as defined in Subsection \ref{subsection: group action}). (We make no freeness hypothesis.) We assume the weights $m_V$ and $m_E$ are $G$-invariant. Let $H=\Delta+q$ be a Schr\"odinger operator on $\Gamma$ defined by a real valuate $G$-invariant potential $q$. Suppose the action of $G$ on the vertex set $V$ of $\Gamma$ has  a finite number of orbits. Let $\lambda_0\in\mathbb R$. The following conditions are equivalent.
	\begin{enumerate}
		\item The operator $H$ satisfies $\lambda_0$-uniqueness on any finite subset of $V$.
		\item There exists a F{\o}lner sequence $\Omega_n$ in $V$ such that  $H$ satisfies $\lambda_0$-uniqueness on each $\Omega_n$.
		\item The spectral projection $E_{\{\lambda_0\}}$ of $H$ is equal to zero. 
		\item The operator $H$ doesn't admit a $\lambda_0$-eigenfunction.
		\item The operator $H$ doesn't admit a $\lambda_0$-eigenfunction with finite support.
		\item The integrated density of states $\lambda\mapsto\tau_1(E_{\lambda})$ of $H$ is continuous at $\lambda_0$.
		
	\end{enumerate}
\end{theo}
\begin{proof} We first show $$(1)\implies(2)\implies(3)\implies(4)\implies(5)\implies(1),$$ then $(3)\iff(6)$.
The implication $(1)\implies(2)$ is obvious. In order to prove $(2)\implies(3)$, assume we have a F{\o}lner sequence $\Omega_n$ such that $H$ satisfies $\lambda_0$-unicity on each $\Omega_n$.	Let us consider the family $\mathcal F\subset\mathcal P(V)$ of all fundamental domains for the action of $G$ on $V$. According to Theorem \ref{theorem: boundaries to bound von Neuman traces},
	\[
		\tau \left(E_{\{\lambda_0\}}\right)\leq\frac{|\partial_2\Omega_n|}{P(\Omega_n,\mathcal F)}.
	\]	
Let us check that the hypothesis of Proposition \ref{proposition: interior points as a lower bound for the packing number} are fulfilled. The group $G$ acts by isometries on  the vertex set $(V,d)$ of $\Gamma$ endowed with its path metric and the action has a finite number of orbits. This implies that the  growth of $(V,d)$ is uniform (define $\mbox{Vol}$ as  the maximal growth over all vertices belonging to a finite fundamental domain for the action of $G$ on $V$). This also implies the existence of  a uniform upper bound $1< r<\infty$  for the inclusive radii relative to  $\mathcal F$:
	\[
		\forall x\in X,\,\exists F_x\in\mathcal F: F_x\subset B(x,r).
	\]
(Choose a fundamental domain $D$ for the action of $G$ on $V$. For each $x\in V$, choose $r_x$ big enough so that $D\subset B(x,r_x)$. As $D$ is finite,  $r=\max_{x\in D}r_x<\infty$. If $g\in G$ then $gD\subset B(gx,r_x)$. This shows that $r$ is a (bad) uniform upper bound for the inclusive radii relative to  $\mathcal F$.) Applying Proposition \ref{proposition: interior points as a lower bound for the packing number}, we obtain that the number of $r$-interior points of any set $\Omega_n$ is bounded in terms of the uniform growth and the packing number:
\[
	|I(\Omega_n,r)|\leq \mbox{Vol}(2r)\cdot P(\Omega_n,\mathcal F).	
\]
The metric space  $(V,d)$ is geodesic in the sense of Definition \ref{definition: geodesic metric with integral values}. Applying Lemma \ref{lemma: folner to bound below interior points} to $\epsilon=1/2$ and to any $\Omega_n$ such that
\begin{equation}\label{equation: condition on Omega_n}
	\mbox{Vol}(r)\cdot|\partial_1\Omega_n|\leq\frac{1}{2}|\Omega_n|,
\end{equation}
we obtain: 
\[
	|I(\Omega_n,r)|\geq\frac{1}{2}|\Omega_n|.
\]
According to Lemma \ref{lemma: comparing boundaries}, 
\[
	|\partial_r\Omega|\leq \mbox{Vol}(r-1)\cdot |\partial_1\Omega|.
\]
Eventually, we come to the conclusion that  
\begin{equation}\label{equation: last inequality}
		\tau\left(E_{\{\lambda_0\}}\right)\leq 2\mbox{Vol}(2r)\mbox{Vol}(1)\frac{|\partial_1\Omega_n|}{|\Omega_n|},	
\end{equation}
providing $\Omega_n$ satisfies Inequality (\ref{equation: condition on Omega_n}).
By definition of a F{\o}lner sequence,
\[
	\lim_{n\to\infty}\frac{|\partial_1\Omega_n|}{|\Omega_n|}=0
\]
Hence Inequality (\ref{equation: condition on Omega_n}) above holds if $n$ is big enough. Letting $n$ goes to infinity in Inequality (\ref{equation: last inequality}) above we deduce that
$\tau\left(E_{\{\lambda_0\}}\right)=0$. According to Subsection \ref{subsection: spectral projections}, the operator $E_{\{\lambda_0\}}$ is positive and belongs to the commutant $\pi(G)'$ because $H\in \pi(G)'$. Hence Proposition
\ref{proposition: properties of the von Neumann trace} applies to the operator $E_{\{\lambda_0\}}$:
\begin{equation}\label{equation: von Neumann definite}
	\tau\left(E_{\{\lambda_0\}}\right)=0 \iff E_{\{\lambda_0\}}=0.	
\end{equation}
This finishes the proof of $(2)\implies(3)$.
As recalled in Subsection \ref{subsection: spectral projections}, condition $(3)$ and condition $(4)$ in the theorem are equivalent. 
The implications $(4)\implies(5)$ and  $(5)\implies(1)$ are obvious. In order to prove the equivalence
$(3)\iff (6)$, recalling the end of Subsection \ref{subsection: spectral projections}, we see that the function
\[
	\lambda\mapsto\tau_1(\lambda)
\]
is continuous at $\lambda_0$ if and only if 
\[
\sum_{x\in D}\langle E_{\{\lambda_0\}}\delta_x,\delta_x\rangle=0,
\]
where $D$ is a fundamental domain for the action of $G$ on $V$. According to Equivalence (\ref{equation: von Neumann definite}) above,  this last condition is equivalent to $E_{\{\lambda_0\}}=0$.
\end{proof}

\end{document}